\documentclass[11pt]{article}
\usepackage{geometry}
\usepackage{amsmath, amssymb, amsthm, dsfont, mathrsfs}
\usepackage{graphicx, caption, float, multirow}

\usepackage{fouriernc, anyfontsize}

\usepackage{fancyhdr}
\usepackage[hidelinks]{hyperref}
\usepackage[dvipsnames]{xcolor}
\hypersetup{colorlinks=true, linkcolor=Orange, citecolor=Orange}

\numberwithin{equation}{section}

\allowdisplaybreaks

\newtheoremstyle{theorem}
  {10pt}
  {10pt}
  {\sl}
  {\parindent}
  {\bf}
  {. }
  { }
  {}
\theoremstyle{theorem}
\newtheorem{theorem}{Theorem}
\newtheorem{corollary}{Corollary}
\newtheorem{lemma}{Lemma}

\begin{document}

\title{\textbf{On the Absolute-Value Integral of a Brownian Motion with Drift: Exact and Asymptotic Formulae}}
\author{Weixuan Xia\thanks{Department of Mathematics, University of Southern California, Los Angeles, CA 90089, USA. \newline\indent Email: \underline{weixuanx@usc.edu}} \and Yuyang Zhang\thanks{Department of Finance, Boston University, Boston, MA 02215, USA. Email: \underline{yyz@bu.edu}}}
\date{2023}
\maketitle
\thispagestyle{plain}

\begin{abstract}
  The present paper is concerned with the integral of the absolute value of a Brownian motion with drift. By establishing an asymptotic expansion of the space Laplace transform, we obtain series representations for the probability density function and cumulative distribution function of the integral, making use of Meijer's G-function. A functional recursive formula is derived for the moments, which is shown to yield only exponentials and Gauss' error function up to arbitrary orders, permitting exact computations. To obtain sharp asymptotic estimates for small- and large-deviation probabilities, we employ a marginal space--time Laplace transform and apply a newly developed generalization of Laplace's method to exponential Airy integrals. The impact of drift on the complete distribution of the integral is explored in depth. The resultant new formulae complement existing ones in the standard Brownian motion case to great extent in terms of both theoretical generality and modeling capacity and have been presented for easy implementation, which numerical experiments demonstrate.
\end{abstract}

\medskip

\newcommand{\dd}{{\rm d}}
\newcommand{\pd}{\partial}
\newcommand{\ii}{{\rm i}}
\newcommand{\PP}{\mathbb{P}}
\newcommand{\E}{\mathbb{E}}
\newcommand{\Var}{\mathrm{Var}}
\newcommand{\Skew}{\mathrm{Skew}}
\newcommand{\EKurt}{\mathrm{EKurt}}
\newcommand{\Ai}{\mathrm{Ai}}
\newcommand{\Bi}{\mathrm{Bi}}
\newcommand{\Gf}{\mathrm{\Gamma}}
\newcommand{\F}{\mathrm{F}}
\newcommand{\G}{\mathrm{G}}
\newcommand{\sech}{\mathrm{sech}}
\newcommand{\Res}{\mathrm{Res}}
\newcommand{\erf}{\mathrm{erf}}

\renewcommand{\Re}{\mathrm{Re}}
\renewcommand{\Im}{\mathrm{Im}}
\renewcommand{\Sigma}{\varSigma}
\renewcommand{\Psi}{\varPsi}
\renewcommand{\Phi}{\varPhi}

\section{Introduction}\label{sec:1}

Studies of the integral of the absolute value of a standard Brownian motion can be traced back to \cite[Kac, 1946]{K1}, which provided an asymptotic series representation for its direct Laplace transform. The result can be seen as a consequence of applying the Kac formula developed in \cite[Kac, 1949]{K2}. Later, \cite[Tak\'{a}cs, 1993]{T1} managed to invert this Laplace transform and provided a rapidly converging series representation for the probability density function, in addition to deriving recursive formulae for the moments. Large-deviation probabilities were subsequently obtained in \cite[Tolmatz, 2003]{T4}. A good number of works have been devoted to studying similar absolute-value Brownian functionals ever since, covering both exact and asymptotic formulae; we refer to \cite[Shepp, 1982]{S1}, \cite[Johnson and Killeen, 1983]{JK}, and \cite[Tolmatz, 2000]{T2} for the Brownian bridge specifically, to \cite[Louchard, 1984]{L1}, \cite[Perman and Wellner, 1996]{PW}, \cite[Janson, 2007]{J1}, and \cite[Janson and Louchard, 2007]{JL} for Brownian excursions and Brownian meanders (including double meanders), and to \cite[Castillo and Boyer, 2016]{CB} for the multidimensional case, among others. Noteworthily, the above works have all emphasized the centeredness of the Brownian motion. Meanwhile, it is appealing to introduce a linear drift component so that the integral functionals are operated on a general continuous L\'{e}vy process; with additional parametrization, the drifted versions are much more adaptive to applications in various directions (discussed below). In such a case, we stress that the drift component will significantly complicate the analysis because it does not scale with the Brownian motion in time.

Let $W\equiv(W_{t})_{t\geq0}$ be a standard Brownian motion defined on a probability space $(\Omega,\mathcal{F},\PP)$ and construct the associated Brownian motion with drift
\begin{equation}\label{1.1}
  \tilde{W}_{t}:=\mu t+\sigma W_{t},\quad t\geq0,
\end{equation}
where the only two parameters, $\mu\in\mathds{R}$ and $\sigma>0$, are referred to as the drift parameter and dispersion parameter, respectively. In this paper we are primarily interested in the process $X\equiv(X_{t})_{t\geq0}$ defined by
\begin{equation}\label{1.2}
  X_{t}:=\int^{t}_{0}|\tilde{W}_{s}|\dd s\equiv\int^{t}_{0}|\mu s+\sigma W_{s}|\dd s,\quad t\geq0,
\end{equation}
which can be viewed as the $\mathrm{L}^{1}$-norm of $\tilde{W}$ taken over time. By construction, $\PP\{X_{0}=0\}=1$ and $X$ has $\PP$-\text{a.s.} positive and increasing sample paths. It is also immediate from (\ref{1.2}) that the sign of $\mu$ has no effect on $X$ up to indistinguishability.

With these fundamental properties, the process $X$ is a natural substitute for the closely related square integral functional $\big(\int^{t}_{0}\tilde{W}^{2}_{s}\dd s\big)_{t\geq0}$, corresponding to the $\mathrm{L}^{2}$-norm, whose distribution has been thoroughly studied in \cite[Xia, 2020]{X} (see also \cite[Tolmatz, 2002]{T2} for the case of a standard Brownian bridge), and the exponential functional $\big(\int^{t}_{0}e^{\tilde{W}_{s}}\dd s\big)_{t\geq0}$, which has a long strand of literature (see, e.g., \cite[Lyasoff, 2016]{L3} for a detailed analysis with the inclusion of drift). Specifically, $X$ finds applications in reliability engineering as a model for the random wear of structural materials \cite[Van Noortwijk, 2009]{vN} and, in financial engineering, when establishing stochastic volatility based on the technique of stochastic time change (\cite[Carr et al., 2003]{CGMY}) or when describing stochastic intensities in a reduced-form credit risk model (\cite[Gourieroux et al., 2006]{GMP}). It is worth noting that in the last aspect, the square integral has recently been adopted by \cite[Chen et al., 2025]{CWX}, which has ensured explicit evaluation of the cumulative intensity under measure changes -- a very useful property that affine jump-diffusion models do not have; in connection with this, the drift component is crucial in generating sizable leverage effects. In comparison, the process $X$ tends to behave less variably than its square counterpart and is more suitable for describing calmer random phenomena for the same purposes. This is our main motivation to study (\ref{1.2}).

This being said, the process $X$ also bears useful connections with various tools for statistical inference. For instance, the standard integral $\int^{1}_{0}|W_{s}|\dd s$ is heavily tied to the so-called ``empirical process'' that appears in goodness-of-fit tests after applying the Khmaladze martingale transformation (\cite[Khmaladze, 1993]{K3}), a technique designed to establish distributional independency when multiple parametric hypotheses are in place. In particular, the distribution of this integral is the (model-free) limit distribution of the corresponding test statistic based on the $\mathrm{L}^{1}$-norm, which is more robust to significant outliers than the commonly adopted statistics based on the $\mathrm{L}^{\infty}$-norm or the $\mathrm{L}^{2}$-norm; see, e.g., \cite[Roberts et al., 2025]{RHS} for illustrations of the latter two. The distribution function will also be explicitly given as a side result of our analysis. In the same direction, the general integral $X_{1}$ can also be identified as the (parameterizable) limit distribution of the conditional log-likelihood function when independent observations have Laplace distributions with normal prior means. We note that likelihood functions of this type sit at the heart of hierarchical matrix factorization for the exponential family; see, e.g., \cite[Wang and Carvalho, 2023]{WC} for a recent development.

On surface, analyzing the distribution of $X_{t}$ (for generic $t>0$) under the parametrization $\{\mu\in\mathds{R},\sigma>0\}$ is much subtler than doing the square functional as no closed-form formula exists for its direct Laplace transform -- not even in the centered case. For this reason, our analysis will also bring about new techniques for handling tail-integrated composite functions and line integrals involving exponential Airy functions. The proposed exact formulae make a decent addition to the collection of distribution formulae for many Brownian functionals in \cite[Borodin and Salminen, 2002]{BS}, while the asymptotic formulae can well be contrasted with those available in \cite[Tolmatz, 2003]{T4}.

\vspace{0.1in}

\section{Laplace transform}\label{sec:2}

To facilitate analysis, we often write $X$ as
\begin{equation}\label{2.1}
  X_{t}=\sigma\int^{t}_{0}|bs+W_{s}|\dd s,\quad t\geq0
\end{equation}
with $b=\mu/\sigma\in\mathds{R}$. In what follows we shall take time $t>0$ as generic and start by considering the space Laplace transform (i.e., with respect to the probability measure),
\begin{equation}\label{2.2}
  \bar{f}_{X_{t}}(u):=\E e^{-uX_{t}},\quad\Re u>0.
\end{equation}
We shall provide an asymptotic formula for $\bar{f}_{X_{t}}(u)$ for $|u|\rightarrow\infty$, which enables us to derive rapidly converging series representations for the corresponding distribution functions. An alternative formula for small $|u|$ will be deduced in Section \ref{sec:3}. Throughout, fractional powers of complex numbers are all taken to be their principal values.

\begin{theorem}\label{thm:1}
Assume $\mu\neq0$. The space Laplace transform (\ref{2.2}) of $X_{t}$ admits the following series representation:
\begin{align}\label{2.3}
  \bar{f}_{X_{t}}(u)&=e^{-\mu^{2}t/(2\sigma^{2})}\sum^{\infty}_{k=1}\sum^{k-1}_{j=0}\frac{\mu^{2j}e^{((\sigma
  u)^{2}/2)^{1/3}t\alpha'_{k-j}}}{(2\sigma^{4}u)^{2j/3}(-\alpha'_{k-j})\Ai(\alpha'_{k-j})(2j)!}\sum^{2j}_{i=0}\binom{2j}{i}(-\alpha'_{k-j})^{2j-i} \nonumber\\
  &\quad\;\times\bigg(\frac{\Gf(i)}{3^{i/3}\Gf(i/3)}+\frac{(-1)^{i}(-\alpha'_{k-j})^{i+1}}{3^{1/3}}
  \bigg(\frac{1}{3^{1/3}\Gf(2/3)(i+1)}{\;_{1}\F_{2}}\bigg(\frac{i+1}{3};\frac{2}{3},\frac{i+4}{3}\bigg|\frac{(\alpha'_{k-j})^{3}}{9}\bigg) \nonumber\\
  &\qquad+\frac{(-\alpha'_{k-j})}{\Gf(1/3)(i+2)}
  {\;_{1}\F_{2}}\bigg(\frac{i+2}{3};\frac{4}{3},\frac{i+5}{3}\bigg|\frac{(\alpha'_{k-j})^{3}}{9}\bigg)\bigg),\quad\Re u>0,
\end{align}
where $\Ai\equiv\Ai(\cdot)$ is the first Airy function and $\{\alpha'_{k}:k\in\mathds{N}_{++}\}\subset\mathds{R}_{--}$ are the zeros of its
derivative, ordered in such a way that $\alpha'_{k}>\alpha'_{k+1}$, $\forall k$, $\Gf\equiv\Gf(\cdot)$ is the usual gamma function and
${\;_{1}\F_{2}}\equiv{\;_{1}\F_{2}}(\cdot;\cdot,\cdot|\cdot)$ is a hypergeometric function,\footnote{We refer to \cite[Abramowitz and Stegun, 1972, \text{Sect.} 10.4]{AS} for detailed properties of the Airy function (with connections to other Bessel-type functions) and to \cite[Slater, 1966]{S2} for that of generalized hypergeometric functions of arbitrary dimensions.} and the limiting value
$\Gf(i)/(3^{i/3}\Gf(i/3))|_{i=0}=1/3$ is understood.
\end{theorem}

\begin{proof}
A direct application of the classical Cameron--Martin--Girsanov theorem allows to rewrite the Laplace transform as
\begin{equation}\label{2.4}
  \bar{f}_{X_{t}}(u)=\E e^{bW_{t}-b^{2}t/2-\sigma uY_{t}},\quad\Re u>0,
\end{equation}
where $Y$ is just $X$ with $\mu=0$ and $\sigma=1$, i.e., the absolute-value integral of $W$.

Denote $v=\sigma u$ for convenience. It is known (see, e.g., \cite[Borodin and Salminen, 2002, \text{Eq.} 1.8.7.(1)]{BS}) that
\begin{equation*}
  \E(e^{-vY_{t}};W_{t}\in\dd
  z)=\sum^{\infty}_{k=1}\frac{v^{1/3}e^{(v^{2}/2)^{1/3}t\alpha'_{k}}\Ai(\alpha'_{k}+(2v)^{1/3}|z|)}{2^{2/3}(-\alpha'_{k})\Ai(\alpha'_{k})}\dd
  z,\quad\Re v>0,\;z\in\mathds{R}.
\end{equation*}
Noting that $\Ai(w)=O(w^{-1/4}e^{-2/3w^{3/2}})$ as $|w|\rightarrow\infty$ for $|\arg w|<\pi/2$ and the above series converges absolutely, by marginalizing $e^{bz}\E(e^{-vY_{t}};W_{t}\in\dd z)$ over $z\in\mathds{R}$ along with the Fubini theorem we have
\begin{align}\label{2.5}
  \bar{f}_{X_{t}}(u)&=\int_{\mathds{R}}e^{bz-b^{2}t/2}\sum^{\infty}_{k=1}\frac{v^{1/3}e^{(v^{2}/2)^{1/3}t\alpha'_{k}}\Ai(\alpha'_{k}+(2v)^{1/3}|z|)}
  {2^{2/3}(-\alpha'_{k})\Ai(\alpha'_{k})}\dd z \nonumber\\
  &=\sum^{\infty}_{k=1}\frac{v^{1/3}e^{(v^{2}/2)^{1/3}t\alpha'_{k}-b^{2}t/2}}
  {2^{2/3}(-\alpha'_{k})\Ai(\alpha'_{k})}\int_{\mathds{R}}e^{bz}\Ai(\alpha'_{k}+(2v)^{1/3}|z|)\dd z.
\end{align}
Since $\Ai$ is an entire function, by the substitutions $\pm(2v)^{1/3}z\mapsto z$ the last integral can be reformatted into
\begin{align*}
  &\quad\;(2v)^{-1/3}\bigg(\int^{\infty}_{\alpha'_{k}}e^{b(2v)^{-1/3}(\alpha'_{k}-z)}\Ai(z)\dd
  z+\int^{\infty}_{\alpha'_{k}}e^{b(2v)^{-1/3}(z-\alpha'_{k})}\Ai(z)\dd z\bigg) \\
  &=2^{2/3}v^{-1/3}\sum^{\infty}_{j=0}(b(2v)^{-1/3})^{2j}\int^{\infty}_{\alpha'_{k}}(z-\alpha'_{k})^{2j}\Ai(z)\dd z,
\end{align*}
where the equality follows from expanding the exponentials around $z=\alpha'_{k}$. Hence, if we define the integrals
\begin{equation}\label{2.6}
  \mathcal{I}_{k,j}:=\int^{\infty}_{\alpha'_{k}}(z-\alpha'_{k})^{2j}\Ai(z)\dd z,\quad k\in\mathds{N}_{++},\;j\in\mathds{N}
\end{equation}
then plugging it into (\ref{2.5}) yields us the formula
\begin{equation*}
  \bar{f}_{X_{t}}(u)=\sum^{\infty}_{k=1}\sum^{\infty}_{j=0}\frac{(b(2v)^{-1/3})^{2j}e^{(v^{2}/2)^{1/3}t\alpha'_{k}-b^{2}t/2}\mathcal{I}_{k,j}}
  {(-\alpha'_{k})\Ai(\alpha'_{k})(2j)!},\quad v=\sigma u,
\end{equation*}
which is equivalent to
\begin{equation}\label{2.7}
  \bar{f}_{X_{t}}(u)=e^{-\mu^{2}t/(2\sigma^{2})}\sum^{\infty}_{k=1}\sum^{k-1}_{j=0}\frac{\mu^{2j}e^{((\sigma
  u)^{2}/2)^{1/3}t\alpha'_{k-j}}\mathcal{I}_{k-j,j}}{(2\sigma^{4}u)^{2j/3}(-\alpha'_{k-j})\Ai(\alpha'_{k-j})(2j)!}.
\end{equation}

To evaluate (\ref{2.6}), we apply binomial expansions to write
\begin{align*}
  \mathcal{I}_{k,j}&=\sum^{2j}_{i=0}\binom{2j}{i}(-\alpha'_{k})^{2j-i}\int^{\infty}_{\alpha'_{k}}z^{i}\Ai(z)\dd z \\
  &=\sum^{2j}_{i=0}\binom{2j}{i}(-\alpha'_{k})^{2j-i}\bigg(\int^{\infty}_{0}+\int^{0}_{\alpha'_{k}}\bigg)z^{i}\Ai(z)\dd z,\quad k\in\mathds{N}_{++},\;j\in\mathds{N},
\end{align*}
where the second equality follows from the continuous differentiability of the integrand with $\alpha'_{k}$'s being all negative. The first integral, corresponding to the Mellin transform of $\Ai$, is well-known (see, e.g., \cite[Vall\'{e}e and Soares, 2010, \text{Eq.} 3.86]{VS}), and we have
\begin{equation*}
  \int^{\infty}_{0}z^{i}\Ai(z)\dd z=\frac{\Gf(i)}{3^{i/3}\Gf(i/3)}
\end{equation*}
with the understanding that $\int^{\infty}_{0}\Ai(z)\dd z=1/3$ if $i=0$. For the second integral, it is enough to look at the primitive. Note that $\Ai$ is directly linked to the (generalized) hypergeometric function via (see \cite[Abramowitz and Stegun, 1972, \text{Sect.} 10.4]{AS})
\begin{equation*}
  \Ai(z)=\frac{1}{3^{2/3}\Gf(2/3)}\;_{0}\F_{1}\bigg(;\frac{2}{3}\bigg|\frac{z^{3}}{9}\bigg) -\frac{z}{3^{1/3}\Gf(1/3)}\;_{0}\F_{1}\bigg(;\frac{4}{3}\bigg|\frac{z^{3}}{9}\bigg),\quad z\in\mathds{C}.
\end{equation*}
Then, as a result of the general Euler transform for hypergeometric functions (see \cite[Slater, 1966, \text{Eq.} 4.1.2]{S2}), we have
\begin{equation*}
  \int z^{i}\Ai(z)\dd z=\frac{z^{i+1}}{3^{2/3}\Gf(2/3)(i+1)}\;_{1}\F_{2}\bigg(\frac{i+1}{3};\frac{2}{3},\frac{i+4}{3}\bigg|\frac{z^{3}}{9}\bigg) -\frac{z^{i+2}}{3^{1/3}\Gf(1/3)(i+2)}\;_{1}\F_{2}\bigg(\frac{i+2}{3};\frac{4}{3},\frac{i+5}{3}\bigg|\frac{z^{3}}{9}\bigg),
\end{equation*}
which can be straight evaluated at the lower limit $\alpha'_{k}$ because $\;_{1}\F_{2}$ is an entire function with the specified parameters, while its vanishing at $z=0$ is clear.

Thus, we have obtained for (\ref{2.6})
\begin{align}\label{2.8}
  \mathcal{I}_{k,j}&=\sum^{2j}_{i=0}\binom{2j}{i}(-\alpha'_{k})^{2j-i}\bigg(\frac{\Gf(i)}{3^{i/3}\Gf(i/3)}+\frac{(-1)^{i}(-\alpha'_{k})^{i+1}}{3^{1/3}} \bigg(\frac{1}{3^{1/3}\Gf(2/3)(i+1)}{\;_{1}\F_{2}}\bigg(\frac{i+1}{3};\frac{2}{3},\frac{i+4}{3}\bigg|\frac{(\alpha'_{k})^{3}}{9}\bigg) \nonumber\\
  &\qquad+\frac{(-\alpha'_{k})}{\Gf(1/3)(i+2)}{\;_{1}\F_{2}}\bigg(\frac{i+2}{3};\frac{4}{3},\frac{i+5}{3}\bigg|\frac{(\alpha'_{k})^{3}}{9}\bigg)\bigg).
\end{align}
Putting (\ref{2.8}) into (\ref{2.5}) and arranging terms we have the desired formula.
\end{proof}

An alternative way to obtain the formula (\ref{2.3}) is by applying the celebrated Feynman--Kac formula and solving a partial differential equation that governs the function $\mathds{R}_{+}\times\mathds{R}\ni(t,z)\mapsto\varphi(t,z):=\E e^{-u\int^{t}_{0}|\mu s+\sigma(W_{s}+z)|\dd s}$. More specifically, this equation reads
\begin{equation*}
  \frac{\pd}{\pd t}\varphi(t,z)=\frac{1}{2}\sigma^{2}\frac{\pd^{2}}{\pd z^{2}}\varphi(t,z)+\mu\frac{\pd}{\pd z}\varphi(t,z)-u|z|\varphi(t,z),
\end{equation*}
subject to the initial condition $\varphi(0,z)=1$, $\forall z\in\mathds{R}$, which can be tackled via standard eigenfunction expansions (see, e.g., \cite[Bell, 1944]{B} for a useful reference). In any case, the adopted method involving the Cameron--Martin--Girsanov theorem is considerably simpler for our purpose.

With the initial value theorem in mind, the series representation (\ref{2.3}) is useful for deducing the required distribution functions of $X_{t}$, which we shall address in the next section. It is, nevertheless, cumbersome for direct implementation (including that of the corresponding characteristic function) as the series is divergent in the limit as $|u|\rightarrow0$; for the same reason, it also cannot be used to analyze the moments of $X_{t}$.

Upon sending $\mu\rightarrow0$ in (\ref{2.3}), the second sum reduces to a single term at $j=0$ and we have the simplified formula
\begin{align}\label{2.9}
  \bar{f}_{X_{t}}(u;\mu=0)&=\sum^{\infty}_{k=1}\frac{e^{((\sigma
  u)^{2}/2)^{1/3}t\alpha'_{k}}}{(-\alpha'_{k})\Ai(\alpha'_{k})}\bigg(\frac{1}{3}+\frac{(-\alpha'_{k})}{3^{1/3}}
  \bigg(\frac{1}{3^{1/3}\Gf(2/3)}{\;_{1}\F_{2}}\bigg(\frac{1}{3};\frac{2}{3},\frac{4}{3}\bigg|\frac{(\alpha'_{k})^{3}}{9}\bigg) \nonumber\\
  &\qquad+\frac{(-\alpha'_{k})}{2\Gf(1/3)}{\;_{1}\F_{2}}\bigg(\frac{2}{3};\frac{4}{3},\frac{5}{3}\bigg|\frac{(\alpha'_{k})^{3}}{9}\bigg)\bigg), \quad\Re u>0,
\end{align}
which is essentially a scaled version of Kac's original discovery in \cite[Kac, 1946]{K1} albeit with explicit evaluation of the Airy integrals.

\vspace{0.2in}

\section{Distribution functions}\label{sec:3}

We proceed to deriving formulae for the distribution functions of $X_{t}$. Recall that the series representation (\ref{2.3}) is absolutely convergent for $\Re u>0$ and hence can be termwise inverted. This automatically guarantees the existence of a probability density function, denoted as
\begin{equation}\label{3.1}
  f_{X_{t}}(x):=\frac{\PP\{X_{t}\in\dd x\}}{\dd x},\quad x>0.
\end{equation}

\begin{theorem}\label{thm:2}
Assume $\mu\neq0$. The probability density function (\ref{3.1}) of $X_{t}$ is given by
\begin{align}\label{3.2}
  f_{X_{t}}(x)&=\frac{e^{-\mu^{2}t/(2\sigma^{2})}}{\sqrt{\pi}x}\sum^{\infty}_{k=1}\sum^{k-1}_{j=0} \bigg(\frac{\mu^{2}t(-\alpha'_{k-j})}{2\sigma^{2}}\bigg)^{j}\frac{1}{3^{j-1/2}(-\alpha'_{k-j})\Ai(\alpha'_{k-j})(2j)!} \nonumber\\
  &\quad\times\sum^{2j}_{i=0}\binom{2j}{i}(-\alpha'_{k-j})^{2j-i}\bigg(\frac{\Gf(i)}{3^{i/3}\Gf(i/3)}+\frac{(-1)^{i}(-\alpha'_{k-j})^{i+1}}{3^{1/3}} \nonumber\\
  &\qquad\times\bigg(\frac{1}{3^{1/3}\Gf(2/3)(i+1)}{\;_{1}\F_{2}}\bigg(\frac{i+1}{3};\frac{2}{3},\frac{i+4}{3}\bigg|\frac{(\alpha'_{k-j})^{3}}{9}\bigg) +\frac{(-\alpha'_{k-j})}{\Gf(1/3)(i+2)}{\;_{1}\F_{2}}\bigg(\frac{i+2}{3};\frac{4}{3},\frac{i+5}{3}\bigg|\frac{(\alpha'_{k-j})^{3}}{9}\bigg)\bigg) \nonumber\\
  &\quad\times\G^{0,3}_{3,2}\bigg(
  \begin{array}{ccccc}
    (j+1)/3 & (j+2)/3 & j/3+1 &  &  \\
     &  &  & 1/2 & 1
  \end{array}
  \bigg|\frac{27x^{2}}{2\sigma^{2}t^{3}(-\alpha'_{k-j})^{3}}\bigg),\quad x>0,
\end{align}
where $\G\equiv\G^{\cdot,\cdot}_{\cdot,\cdot}(\cdots|\cdot)$ denotes Meijer's G-function.\footnote{We refer to \cite[Erd\'{e}lyi et al., 1953, \text{Chap.} V]{EMOT} for various definitions and detailed properties of the G-function.}
\end{theorem}

\begin{proof}
Based on Theorem \ref{thm:1}, inverting (\ref{2.3}) comes down to inverting $u^{-2j/3}e^{-u^{2/3}}$ in $u$ for $j\geq0$. Since the power of the exponent is rational we can consult the general formula in \cite[Prudnikov et al., 1992, \text{Eq.} 2.2.1.19]{PBMb} to obtain that, for arbitrary $c>0$,
\begin{equation}\label{3.3}
  \frac{1}{2\pi\ii}\int^{c+\ii\infty}_{c-\ii\infty}e^{ux-u^{2/3}}u^{-2j/3}\dd u=\sqrt{\frac{3}{\pi}}\frac{1}{x}\bigg(\frac{x}{2}\bigg)^{2j/3}\G^{3,0}_{2,3}\bigg(
  \begin{array}{ccccc}
     &  &  & j/3 & j/3+1/2 \\
    0 & 1/3 & 2/3 &  &
  \end{array}
  \bigg|\frac{4}{27x^{2}}\bigg),\quad x>0.
\end{equation}
This expression can be shortened by using some fundamental identities involving the G-function (in particular, see \cite[Erd\'{e}lyi et al., 1953, \text{Eqs.} 5.3.1.(8) and 5.3.1.(9)]{EMOT}); we then have
\begin{equation}\label{3.4}
  \frac{1}{2\pi\ii}\int^{c+\ii\infty}_{c-\ii\infty}e^{ux-u^{2/3}}u^{-2j/3}\dd u=\frac{1}{3^{j-1/2}\sqrt{\pi}x}\G^{0,3}_{3,2}\bigg(
  \begin{array}{ccccc}
    (j+1)/3 & (j+2)/3 & j/3+1 &  &  \\
     &  &  & 1/2 & 1
  \end{array}
  \bigg|\frac{27x^{2}}{4}\bigg).
\end{equation}
Using (\ref{3.4}) along with the scaling property of the Laplace transform, the density function of $X_{t}$ can be expressed as
\begin{align}\label{3.5}
  f_{X_{t}}(x)&=\frac{e^{-\mu^{2}t/(2\sigma^{2})}}{\sqrt{\pi}x}\sum^{\infty}_{k=1}\sum^{k-1}_{j=0} \bigg(\frac{\mu^{2}t(-\alpha'_{k-j})}{2\sigma^{2}}\bigg)^{j} \frac{\mathcal{I}_{k-j,j}}{3^{j-1/2}(-\alpha'_{k-j})\Ai(\alpha'_{k-j})(2j)!} \nonumber\\
  &\quad\times\G^{0,3}_{3,2}\bigg(
  \begin{array}{ccccc}
    (j+1)/3 & (j+2)/3 & j/3+1 &  &  \\
     &  &  & 1/2 & 1
  \end{array}
  \bigg|\frac{27x^{2}}{2\sigma^{2}t^{3}(-\alpha'_{k-j})^{3}}\bigg),\quad x>0,
\end{align}
after simplifications, where $\mathcal{I}_{k,j}$ is given by (\ref{2.6}). Using (\ref{2.8}) to explicitly rewrite $\mathcal{I}_{k,j}$ in (\ref{3.5}) delivers (\ref{3.2}).
\end{proof}

In the limit as $\mu\rightarrow0$, with the second series in (\ref{3.2}) reduced to a single term at $j=0$ and the specialization that (see (\ref{3.3}) and \cite[Erd\'{e}lyi et al., 1953, \text{Eqs.} 5.3.1.(7) and 5.6.(5)]{EMOT})
\begin{align*}
  &\quad\G^{0,3}_{3,2}\bigg(
  \begin{array}{ccccc}
    1/3 & 2/3 & 1 &  &  \\
     &  &  & 1/2 & 1
  \end{array}
  \bigg|\frac{27x^{2}}{2\sigma^{2}t^{3}(-\alpha'_{k})^{3}}\bigg) \\
  &=\G^{0,2}_{2,1}\bigg(
  \begin{array}{cccc}
    1/3 & 2/3 &  &  \\
     &  &  & 1/2
  \end{array}
  \bigg|\frac{27x^{2}}{2\sigma^{2}t^{3}(-\alpha'_{k})^{3}}\bigg) \\
  &=\frac{\sigma^{4/3}t(-\alpha'_{k})}{9\times2^{2/3}\sqrt{\pi}x^{4/3}}\bigg(\Gf\bigg(\frac{1}{3}\bigg)t(-\alpha'_{k}) \;_{1}\F_{1}\bigg(\frac{7}{6};\frac{4}{3}\bigg|\frac{2\sigma^{2}t^{3}(\alpha'_{k})^{3}}{27x^{2}}\bigg) -2^{1/3}\Gf\bigg(-\frac{1}{3}\bigg)\bigg(\frac{x}{\sigma}\bigg)^{2/3} \;_{1}\F_{1}\bigg(\frac{5}{6};\frac{2}{3}\bigg|\frac{2\sigma^{2}t^{3}(\alpha'_{k})^{3}}{27x^{2}}\bigg)\bigg)
\end{align*}
in terms of Kummer's confluent hypergeometric function $\;_{1}\F_{1}\equiv\;_{1}\F_{1}(\cdot;\cdot|\cdot)$,\footnote{Specifically, this function is used in this context for notational consistency with other hypergeometric functions, in addition to computational advantages.} we have the following reduced formula after rearrangement:
\begin{align}\label{3.6}
  f_{X_{t}}(x;\mu=0)&=\frac{\sigma^{4/3}t}{\sqrt{3}\times2^{2/3}\pi x^{7/3}}\sum^{\infty}_{k=1}\frac{1}{\Ai(\alpha'_{k})} \bigg(\frac{1}{3}+\frac{(-\alpha'_{k})}{3^{1/3}} \bigg(\frac{1}{3^{1/3}\Gf(2/3)}{\;_{1}\F_{2}}\bigg(\frac{1}{3};\frac{2}{3},\frac{4}{3}\bigg|\frac{(\alpha'_{k})^{3}}{9}\bigg) \nonumber\\
  &\qquad+\frac{(-\alpha'_{k})}{2\Gf(2/3)}{\;_{1}\F_{2}}\bigg(\frac{2}{3};\frac{4}{3},\frac{5}{3}\bigg|\frac{(\alpha'_{k})^{3}}{9}\bigg)\bigg) \bigg(\Gf\bigg(\frac{4}{3}\bigg)t(-\alpha'_{k})\;_{1}\F_{1}\bigg(\frac{7}{6};\frac{4}{3}\bigg|\frac{2\sigma^{2}t^{3}(\alpha'_{k})^{3}}{27x^{2}}\bigg) \nonumber\\
  &\qquad+2^{1/3}\Gf\bigg(\frac{2}{3}\bigg)\bigg(\frac{x}{\sigma}\bigg)^{2/3} \;_{1}\F_{1}\bigg(\frac{5}{6};\frac{2}{3}\bigg|\frac{2\sigma^{2}t^{3}(\alpha'_{k})^{3}}{27x^{2}}\bigg)\bigg),\quad x>0.
\end{align}
If one uses instead Tricomi's confluent hypergeometric function which is a modified form of $\;_{1}\F_{1}$ in (\ref{3.6}), this becomes a scaled version of the original formula obtained in \cite[Tak\'{a}cs, 1993, \text{Thm.} 1]{T1}. The same result can be deduced by directly inverting (\ref{2.9}).

In a similar fashion we can obtain the next result for the cumulative distribution function,
\begin{equation}\label{3.7}
  F_{X_{t}}(x):=\PP\{X_{t}\leq x\}=\int^{x}_{0}f_{X_{t}}(y)\dd y,\quad x>0.
\end{equation}

\begin{theorem}\label{thm:3}
Assume $\mu\neq0$. The cumulative distribution function (\ref{3.7}) of $X_{t}$ is given by
\begin{align}\label{3.8}
  F_{X_{t}}(x)&=\frac{\sigma\sqrt{t^{3}}e^{-\mu^{2}t/(2\sigma^{2})}}{\sqrt{2\pi}x}\sum^{\infty}_{k=1}\sum^{k-1}_{j=0} \bigg(\frac{\mu^{2}t(-\alpha'_{k-j})}{2\sigma^{2}}\bigg)^{j} \frac{\sqrt{-\alpha'_{k-j}}}{3^{j+1}\Ai(\alpha'_{k-j})(2j)!}\sum^{2j}_{i=0}\binom{2j}{i}(-\alpha'_{k-j})^{2j-i} \nonumber\\
  &\qquad\times\bigg(\frac{\Gf(i)}{3^{i/3}\Gf(i/3)}+\frac{(-1)^{i}(-\alpha'_{k-j})^{i+1}}{3^{1/3}} \bigg(\frac{1}{3^{1/3}\Gf(2/3)(i+1)}{\;_{1}\F_{2}}\bigg(\frac{i+1}{3};\frac{2}{3},\frac{i+4}{3}\bigg|\frac{(\alpha'_{k-j})^{3}}{9}\bigg) \nonumber\\
  &\qquad+\frac{(-\alpha'_{k-j})}{\Gf(1/3)(i+2)} {\;_{1}\F_{2}}\bigg(\frac{i+2}{3};\frac{4}{3},\frac{i+5}{3}\bigg|\frac{(\alpha'_{k-j})^{3}}{9}\bigg)\bigg) \nonumber\\
  &\quad\times\G^{0,3}_{3,2}\bigg(
  \begin{array}{ccccc}
    j/3+5/6 & j/3+7/6 & j/3+3/2 &  &  \\
     &  &  & 1/2 & 1
  \end{array}
  \bigg|\frac{27x^{2}}{2\sigma^{2}t^{3}(-\alpha'_{k-j})^{3}}\bigg),\quad x>0.
\end{align}
\end{theorem}

\begin{proof}
Using the integration property of the Laplace transform that
\begin{equation*}
  \bar{F}_{X_{t}}(u):=\int^{\infty}_{0}e^{-ux}F_{X_{t}}(x)\dd x=\frac{\bar{f}_{X_{t}}(u)}{u},\quad\Re u>0,
\end{equation*}
we are led to invert $u^{-2j/3-1}e^{-u^{2/3}}$ in light of (\ref{2.3}). Hence, following the proof of Theorem \ref{thm:2}, we only need to replace $j$ with $j+3/2$ in (\ref{3.4}) to obtain the result after arranging terms.
\end{proof}

Similarly, upon sending $\mu\rightarrow0$ in (\ref{3.8}) we immediately have the reduced formula
\begin{align}\label{3.9}
  F_{X_{t}}(x;\mu=0)&=\frac{\sigma\sqrt{t^{3}}}{3\sqrt{2\pi}x}\sum^{\infty}_{k=1}\frac{\sqrt{-\alpha'_{k}}}{\Ai(\alpha'_{k})} \bigg(\frac{1}{3}+\frac{(-\alpha'_{k})}{3^{1/3}} \bigg(\frac{1}{3^{1/3}\Gf(2/3)}{\;_{1}\F_{2}}\bigg(\frac{1}{3};\frac{2}{3},\frac{4}{3}\bigg|\frac{(\alpha'_{k})^{3}}{9}\bigg) \nonumber\\
  &\quad\;+\frac{(-\alpha'_{k})}{2\Gf(1/3)}{\;_{1}\F_{2}}\bigg(\frac{2}{3};\frac{4}{3},\frac{5}{3}\bigg|\frac{(\alpha'_{k})^{3}}{9}\bigg)\bigg) \G^{0,3}_{3,2}\bigg(
  \begin{array}{ccccc}
    5/6 & 7/6 & 3/2 &  &  \\
     &  &  & 1/2 & 1
  \end{array}
  \bigg|\frac{27x^{2}}{2\sigma^{2}t^{3}(-\alpha'_{k})^{3}}\bigg),\quad x>0.
\end{align}
Here the G-function cannot be simplified to more elementary functions due to irreducible dimensionality, which is most likely why no similar result was given in \cite[Tak\'{a}cs, 1993]{T1}.

As a remark, pertaining to the Khmaladze transformation-based goodness-of-fit tests mentioned in Section \ref{sec:1}, the new formula (\ref{3.9}), taken with $\sigma=t=1$, determines the exact null distribution of the test statistic based on the $\mathrm{L}^{1}$-norm.

\vspace{0.2in}

\section{Moments}\label{sec:4}

Let the moments of $X_{t}$ be denoted by
\begin{equation}\label{4.1}
  M_{X_{t}}(n):=\E X^{n}_{t},\quad n\in\mathds{N},
\end{equation}
all of which are positive. Their existence is established in the next lemma.

\begin{lemma}\label{lem:1}
We have that $M_{X_{t}}(n)<\infty$, $\forall n\in\mathds{N}$.
\end{lemma}

\begin{proof}
Let $A_{s}=\{\omega\in\Omega:|\tilde{W}_{s}|\leq1\}$ for $s\in[0,t]$, recalling (\ref{1.1}). An application of the $C^{p}$-inequality gives
\begin{align*}
  M_{X_{t}}(n)&=\E\bigg(\int^{t}_{0}|\tilde{W}_{s}|\mathds{1}_{A_{s}}\dd s+\int^{t}_{0}|\tilde{W}_{s}|\mathds{1}_{A^{\complement}_{s}}\dd s\bigg)^{n} \\
  &\leq2^{(n-1)^{+}}\bigg(\E\bigg(\int^{t}_{0}|\tilde{W}_{s}|\mathds{1}_{A_{s}}\dd s\bigg)^{n}+\E\bigg(\int^{t}_{0}|\tilde{W}_{s}|\mathds{1}_{A^{\complement}_{s}}\dd s\bigg)^{n}\bigg) \\
  &\leq2^{(n-1)^{+}}\bigg(t^{n}+\E\bigg(\int^{t}_{0}(\tilde{W}_{s})^{2}\dd s\bigg)^{n}\bigg) \\
  &<\infty,
\end{align*}
where $(\cdot)^{+}$ means the positive part and the last inequality follows from the finiteness of the expectation (i.e., the moments of the square integral of Brownian motions with drft); see, e.g., \cite[Xia, 2020, \text{Thm.} 3]{X}.
\end{proof}

Beside the zeroth moment obviously equal to 1, the first moment (mean) can be easily computed by applying the Fubini theorem. Since $\tilde{W}_{s}$ is for every fixed $s$ normally distributed with mean $\mu s$ and variance $\sigma^{2}s$, we obtain in particular
\begin{equation}\label{4.2}
  M_{X_{t}}(1)\equiv\E X_{t}=\int^{t}_{0}\E|\tilde{W}_{s}|\dd s=\frac{\sigma\sqrt{t}(\mu^{2}t-\sigma^{2})}{\sqrt{2\pi}\mu^{2}}e^{-\mu^{2}t/(2\sigma^{2})} +\frac{\mu^{4}t^{2}+\sigma^{4}}{2\mu^{3}}\erf\frac{\mu\sqrt{t}}{\sqrt{2}\sigma},
\end{equation}
where $\erf\equiv\erf(\cdot)$ stands for Gauss' error function. However, this approach quickly fails for moments of higher orders when multiple integration has to be carried out -- in fact, note that with the second moment no closed-form expression exists for the expectation $\E|\tilde{W}_{s}\tilde{W}_{s'}|$, $s,s'\in[0,t]$.

In an attempt to derive an alternative series representation for the Laplace transform $\bar{f}_{X_{t}}$ to (\ref{2.3}), in the next theorem we are able to provide a general formula that allows exact computations of the moments of arbitrary orders.

\begin{theorem}\label{thm:4}
Assume $\mu\neq0$. The moments (\ref{4.1}) of $X_{t}$ are given by
\begin{equation}\label{4.3}
  M_{X_{t}}(n)=\frac{(-1)^{n}\sigma^{n}t^{3n/2}n!e^{-\mu^{2}t/(2\sigma^{2})}}{\sqrt{2^{n}}}\sum^{n}_{k=0}\tilde{c}^{(1)}_{n-k} \sum^{k}_{l=0}\sum^{l}_{i=0}\binom{l-k-1/2}{l-i}\bigg(\frac{3}{2}\bigg)^{l}c^{(2)}_{k-l}\Sigma\bigg(n,l,i\bigg|\frac{\mu^{2}t}{2\sigma^{2}}\bigg), \quad n\in\mathds{N},
\end{equation}
where the series coefficients satisfy
\begin{align}\label{4.4}
  &c^{(2)}_{k}=\frac{(-1)^{k}\Gf(k+5/6)\Gf(k+1/6)(3/4)^{k}}{2\pi k!},\quad c^{(1)}_{k}=\frac{(6k+1)}{(1-6k)}c^{(2)}_{k},\quad k\in\mathds{N}, \nonumber\\
  &\tilde{c}^{(1)}_{0}=1\rightsquigarrow\tilde{c}^{(1)}_{k}=-\sum^{k}_{m=1}c^{(1)}_{m}\tilde{c}^{(1)}_{k-m},\;\;k\geq1
\end{align}
and for $0\leq i\leq l\leq n$, $\Sigma(n,l,i;\mu^{2}t/(2\sigma^{2}))=S_{i}(1)$ are determined by the differential--difference equation
\begin{align}\label{4.5}
  &S_{0}(p)=\frac{l!}{\Gf(3n/2+1)}\;_{2}\F_{2}\bigg(\frac{l+1}{2},\frac{l}{2}+1;\frac{1}{2},\frac{3n}{2}+1\bigg|\frac{\mu^{2}t}{2\sigma^{2}p}\bigg) \nonumber\\
  \rightsquigarrow&S_{i}(p)=\frac{\Gf(2/3)}{i}\sum^{i}_{m=1}\frac{(m-i)S_{i-m}(p)-2mpS'_{i-m}(p)}{\Gf(2/3-m)(m+1)!},\quad \Re p>0,
\end{align}
with $\;_{2}\F_{2}(\cdot,\cdot;\cdot,\cdot|\cdot)$ being a generalized hypergeometric function.
\end{theorem}

\begin{proof}
Based on (\ref{2.1}), we rewrite
\begin{equation}\label{4.6}
  M_{X_{t}}(n)=\sigma^{n}e^{-b^{2}t/2}\E\big(e^{bW_{t}}Y^{n}_{t}\big),\quad n\in\mathds{N}
\end{equation}
with $b=\mu/\sigma$, assumed nonzero.

By consulting a known formula for the space--time Laplace transform of the distribution of $Y_{t}$ (see \cite[Tolmatz, 2000, \text{Eq.} (2)]{T1}) stemming from applying the result of \cite[Kac, 1949]{K2}, we have for $v=\sigma u$ with positive real part that
\begin{equation*}
  \int^{\infty}_{0}e^{-yt}\E\big(e^{bW_{t}-vY_{t}};W_{t}\in\dd z\big)\dd t=-\frac{e^{bz}\Ai(2^{1/3}v^{-2/3}(y+v|z|))}{(2v)^{1/3}\Ai'(2^{1/3}v^{-2/3}y)}\dd z,\quad\Re y>0,\;z\in\mathds{R}.
\end{equation*}
Marginalization yields
\begin{align}\label{4.7}
  \phi(u,y):=\int^{\infty}_{0}e^{-(y-b^{2}/2)t}\bar{f}_{X_{t}}(u)\dd t&=-\int_{\mathds{R}}\frac{e^{bz}\Ai(2^{1/3}v^{-2/3}(y+v|z|))}{(2v)^{1/3}\Ai'(2^{1/3}v^{-2/3}y)}\dd z \nonumber\\
  &=-\frac{1}{2^{1/3}v^{4/3}\Ai'(2^{1/3}v^{-2/3}y)}\bigg(e^{-by/v}\int^{\infty}_{y}e^{bz/v}\Ai(2^{1/3}v^{-2/3}z)\dd z \nonumber\\
  &\quad+e^{by/v}\int^{\infty}_{y}e^{-bz/v}\Ai(2^{1/3}v^{-2/3}z)\dd z\bigg),\quad v=\sigma u,\;\Re y>\frac{b^{2}}{2},
\end{align}
where the second equality follows from splitting the integral with respect to $|z|$ and the substitutions $y\pm vz\mapsto z$. From here we write $a=2^{1/3}v^{-2/3}$ with $\Re a>0$ and $\beta=b/v\neq0$.

For the (first) Airy function and its derivative the following asymptotic expansions are well-known (see \cite[Abramowitz and Stegun, 1972, \text{Eqs.} 10.4.59 and 10.4.61]{AS}):
\begin{equation}\label{4.8}
  \Ai(ay)=\frac{e^{-2/3(ay)^{3/2}}}{2\sqrt{\pi}}\sum^{\infty}_{k=0}\frac{c^{(2)}_{k}}{(ay)^{3k/2+1/4}},\quad \Ai'(ay)=-\frac{e^{-2/3(ay)^{3/2}}}{2\sqrt{\pi}}\sum^{\infty}_{k=0}\frac{c^{(1)}_{k}}{(ay)^{3k/2-1/4}},
\end{equation}
where $c^{(1)}_{k}$'s and $c^{(2)}_{k}$'s are as specified in (\ref{4.4}). Plugging the first expansion in (\ref{4.7}) we obtain
\begin{align*}
  A&:=e^{-\beta y}\int^{\infty}_{y}e^{\beta z}\Ai(az)\dd z+e^{\beta y}\int^{\infty}_{y}e^{-\beta z}\Ai(az)\dd z \\
  &=\frac{1}{2\sqrt{\pi}}\sum^{\infty}_{k=0}\frac{c^{(2)}_{k}}{a^{1/4+3k/2}}\bigg(e^{-\beta y}\int^{\infty}_{y}e^{\beta z-2/3(az)^{3/2}}z^{-1/4-3k/2}\dd z+e^{\beta y}\int^{\infty}_{y}e^{-\beta z-2/3(az)^{3/2}}z^{-1/4-3k/2}\dd z\bigg).
\end{align*}
As the integrands are analytic in the right half-plane, the substitution $z^{3/2}\mapsto z$ gives
\begin{align*}
  A=\frac{1}{3\sqrt{\pi}}\sum^{\infty}_{k=0}\frac{c^{(2)}_{k}}{a^{1/4+3k/2}}\bigg(e^{-\beta y}\int^{\infty}_{y^{3/2}}e^{\beta z^{2/3}-2/3a^{3/2}z}z^{-1/2-k}\dd z+e^{\beta y}\int^{\infty}_{y^{3/2}}e^{-\beta z^{2/3}-2/3a^{3/2}z}z^{-1/2-k}\dd z\bigg).
\end{align*}
The asymptotic expansions of the above integrals can be derived via integration-by-parts, i.e.,
\begin{align*}
  e^{-\beta y}\int^{\infty}_{y^{3/2}}e^{\beta z^{2/3}-2/3a^{3/2}z}z^{-1/2-k}\dd z&=-\frac{3}{2}a^{-3/2}\bigg(\big(e^{\beta z^{2/3}-2/3a^{3/2}z}z^{-1/2-k}\big)\big|^{\infty}_{z=y^{3/2}} \\
  &\qquad+\frac{3}{2}a^{-3/2}\int^{\infty}_{y^{3/2}}e^{-2/3a^{3/2}z}\frac{\dd}{\dd z}\big(e^{\beta z^{2/3}}z^{-1/2-k}\big)\dd z=\cdots
\end{align*}
and similarly for the other integral. Repeating indefinitely we obtain
\begin{equation}\label{4.9}
  A=\frac{e^{-2/3(ay)^{3/2}}}{2\sqrt{\pi}}\sum^{\infty}_{k=0}\frac{c^{(2)}_{k}}{a^{7/4+3k/2}}\sum^{\infty}_{j=0}\frac{g^{(j)}_{+,k}(y^{3/2}) +g^{(j)}_{-,k}(y^{3/2})}{(2/3a^{3/2})^{j}},
\end{equation}
where
\begin{equation*}
  g^{(j)}_{+,k}(z)=e^{-\beta z^{2/3}}\frac{\dd^{j}}{\dd z^{j}}(e^{\beta z^{2/3}}z^{-1/2-k})\quad\text{and}\quad g^{(j)}_{-,k}(z)=e^{\beta z^{2/3}}\frac{\dd^{j}}{\dd z^{j}}(e^{-\beta z^{2/3}}z^{-1/2-k}).
\end{equation*}
Then, an application of Fa\`{a} di Bruno's formula allows to express these repeated derivatives in terms of partial Bell polynomials (see, e.g., \cite[Johnson, 2002]{J2}), namely
\begin{align*}
  g^{(j)}_{+,k}(y^{3/2})&=\sum^{j}_{i=0}\binom{j}{i}\frac{\Gf(1/2-k)}{\Gf(1/2-k-(j-i))}y^{-3/4-3(k+(j-i))/2} \\
  &\qquad\times\sum^{i}_{l=0}\beta^{l}\mathbf{B}_{i,l}\bigg(\frac{\Gf(5/3)}{\Gf(5/3-m)}y^{1-3m/2}\bigg|m\in\mathds{N}\cap[1,i-l+1]\bigg),
\end{align*}
where
\begin{equation}\label{4.10}
  \mathbf{B}_{i,l}\bigg(\frac{\Gf(5/3)}{\Gf(5/3-m)}y^{1-3m/2}\bigg|m\in\mathds{N}\cap[1,i-l+1]\bigg)=\sum_{\iota}\frac{i!} {\prod^{i-l+1}_{m=1}\iota_{m}!}\prod^{i-l+1}_{m=1}\frac{(\Gf(5/3)/\Gf(5/3-m)y^{1-3m/2})^{\iota_{m}}}{m!},
\end{equation}
where the sum runs over all such combinations $\{\iota_{m}:m\in\mathds{N}\cap[1,i-l+1]\}$ that $\sum^{i+l-1}_{m=1}\iota_{m}=l$ and $\sum^{i+l-1}_{m=1}m\iota_{m}=i$. The same holds for $g^{(j)}_{-,k}(y^{3/2})$ except that $\beta$ is replaced by $-\beta$.

Straightforward manipulation then gives
\begin{align*}
  &\quad\; g^{(j)}_{+,k}(y^{3/2})+g^{(j)}_{-,k}(y^{3/2}) \\
  &=2\sum^{j}_{i=0}\binom{j}{i}\frac{\Gf(1/2-k)}{\Gf(1/2-k-(j-i))}y^{-3/4-3(k+(j-i))/2} \\
  &\qquad\times\sum^{i}_{l=0}\beta^{l}\mathds{1}_{\text{even}}(l)\mathbf{B}_{i,l} \bigg(\frac{\Gf(5/3)}{\Gf(5/3-m)}y^{1-3m/2}\bigg|m\in\mathds{N}\cap[1,i-l+1]\bigg) \\
  &=2\sum^{j}_{l=0}\beta^{j-l}\mathds{1}_{\text{even}}(j-l)\sum^{l}_{i=0}\binom{j}{i+j-l}\mathbf{B}_{j+i-l,j-l} \bigg(\frac{\Gf(5/3)}{\Gf(5/3-m)}y^{1-3m/2}\bigg|m\in\mathds{N}\cap[1,i+1]\bigg) \\
  &\qquad\times\frac{\Gf(1/2-k)}{\Gf(1/2-k-(l-i))}y^{-3/4-3(k+(l-i))/2}.
\end{align*}
By substituting this into (\ref{4.9}) and changing the order of summation (with respect to $j$ and $l$) we have
\begin{align}\label{4.11}
  A&=\frac{2^{5/12}v^{7/6}e^{-2/3(ay)^{3/2}}}{2\sqrt{\pi}}\sum^{\infty}_{k=0}\frac{c^{(2)}_{k}}{2^{k/2}}\sum^{\infty}_{j=0}\sum^{j}_{l=0} \frac{3^{j}v^{k+l}b^{j-l}\mathds{1}_{\text{even}}(j-l)}{2^{3j/2}} \nonumber\\
  &\quad\times\sum^{l}_{i=0}\binom{j}{i+j-l}\mathbf{B}_{j+i-l,j-l} \bigg(\frac{\Gf(5/3)}{\Gf(5/3-m)}y^{1-3m/2}\bigg|m\in\mathds{N}\cap[1,i+1]\bigg)\frac{\Gf(1/2-k)}{\Gf(1/2-k-(l-i))}y^{-3/4-3(k+(l-i))/2} \nonumber\\
  &=\frac{2^{5/12}v^{7/6}e^{-2/3(ay)^{3/2}}}{2\sqrt{\pi}} \sum^{\infty}_{n=0}v^{n}\sum^{n}_{l=0}\sum^{\infty}_{j=0}\frac{c^{(2)}_{n-l}}{2^{(n-l)/2}}\frac{3^{j+l}b^{j} \mathds{1}_{\text{even}}(j)}{2^{3(j+l)/2}} \nonumber\\
  &\quad\times\sum^{l}_{i=0}\binom{j+l}{j+i}\mathbf{B}_{j+i,j}\bigg(\frac{\Gf(5/3)}{\Gf(5/3-m)}y^{1-3m/2}\bigg|m\in\mathds{N}\cap[1,i+1]\bigg) \frac{\Gf(1/2-n+l)}{\Gf(1/2-n+i)}y^{-3/4-3(n-i)/2}.
\end{align}

Thus, if we divide (\ref{4.11}) by the product of $-2^{1/3}v^{4/3}$ and $\Ai'(ay)$ as in (\ref{4.8}), we obtain an asymptotic expansion of $\phi(u,y)$ in $y$ involving only nonnegative integer powers of $v=\sigma u$. More specifically, by applying Cauchy's product rule and arranging terms we have
\begin{align*}
  \phi(u,y)&=\sum^{\infty}_{n=0}v^{n}\sum^{n}_{k=0}\frac{\tilde{c}^{(1)}_{n-k}}{2^{(n-k)/2}y^{3(n-k)/2}} \sum^{k}_{l=0}\sum^{\infty}_{j=0}\frac{c^{(2)}_{k-l}}{2^{(k-l)/2}}\frac{3^{j+l}b^{j}\mathds{1}_{\text{even}}(j)}{2^{3(j+l)/2}} \\
  &\quad\times\sum^{l}_{i=0}\binom{j+l}{j+i}\mathbf{B}_{j+i,j}\bigg(\frac{\Gf(5/3)}{\Gf(5/3-m)}y^{1-3m/2}\bigg|m\in\mathds{N}\cap[1,i+1]\bigg) \\
  &\qquad\times\frac{\Gf(1/2-k+l)}{\Gf(1/2-k+i)}y^{-1-3(k-i)/2}.
\end{align*}
With $\Ai'$ in the denominator, it is familiar (\cite[Gradshteyn and Ryzhik, 2007, \text{Eq.} 0.313]{GR}) that the reciprocal series coefficients $\tilde{c}^{(1)}_{k}$'s satisfy the recurrence relation in (\ref{4.4}).

On the other hand, Lemma \ref{lem:1} implies the canonical expansion\footnote{This expansion has a positive radius of convergence according to the proof of Lemma \ref{lem:1}.}
\begin{equation}\label{4.12}
  \bar{f}_{X_{t}}(u)=\sum^{\infty}_{n=0}\frac{M_{X_{t}}(n)(-u)^{n}}{n!}.
\end{equation}
Since (\ref{4.7}) is the time Laplace transform of $e^{b^{2}t/2}\bar{f}_{X_{t}}(u)$ for fixed $u$, as a result of matching the powers of $v$ with reference to (\ref{4.6}) we must have
\begin{align}\label{4.13}
  \int^{\infty}_{0}e^{-(y-b^{2}/2)t}M_{X_{t}}(n)\dd t&=\frac{(-1)^{n}\sigma^{n}n!}{\sqrt{2^{n}}}\sum^{n}_{k=0}\tilde{c}^{(1)}_{n-k} \sum^{k}_{l=0}\sum^{l}_{i=0}\frac{3^{l}c^{(2)}_{k-l}\Gf(1/2-k+l)}{2^{l}\Gf(1/2-k+i)} \sum^{\infty}_{j=0}\binom{2j+l}{2j+i}\bigg(\frac{9b^{2}}{8}\bigg)^{j} \nonumber\\
  &\quad\;\times\mathbf{B}_{2j+i,2j}\bigg(\frac{\Gf(5/3)}{\Gf(5/3-m)}y^{1-3m/2}\bigg|m\in\mathds{N}\cap[1,i+1]\bigg)y^{-1-3(n-i)/2}
\end{align}
after simplifications. Observe from (\ref{4.10}) that the right side of (\ref{4.13}) only has negative half-integer powers of $y$.

To extract the powers of $y$ contained in the Bell polynomials, we employ a regularization argument (with infinitesimal $\epsilon>0$) that for $i,j\in\mathds{N}$,
\begin{align}\label{4.14}
  \mathbf{B}_{2j+i,2j}\bigg(\frac{\Gf(5/3)}{\Gf(5/3-m)}y^{1-3m/2}\bigg|m\in\mathds{N}\cap[1,i+1]\bigg) &=\binom{2j+i}{2j}\frac{\dd^{i}}{\dd\epsilon^{i}}\Bigg(\sum^{\infty}_{m=1}\frac{\Gf(5/3)y^{1-3m/2}\epsilon^{m-1}}{\Gf(5/3-m)m!}\Bigg)^{2j} \Bigg|_{\epsilon\rightarrow0} \nonumber\\
  &=\binom{2j+i}{2j}\bigg(\frac{4}{9}\bigg)^{j}i!P_{i}(j)y^{-j-3i/2}.
\end{align}
With an infinite radius of convergence, the above power series has induced ($y$-independent) polynomials of degree $i$ that familiarly satisfy the recurrence relation (\cite[Gradshteyn and Ryzhik, 2007, \text{Eq.} 0.314]{GR})
\begin{equation}\label{4.15}
  P_{0}(j)=1\rightsquigarrow P_{i}(j)=\frac{\Gf(2/3)}{i}\sum^{i}_{m=1}\frac{2jm+m-i}{\Gf(2/3-m)(m+1)!}P_{i-m}(j),\quad i\geq1,
\end{equation}
with all rational coefficients.

Plugging (\ref{4.14}) and (\ref{4.15}) back into (\ref{4.13}) and simplifying we have
\begin{align}\label{4.16}
  \int^{\infty}_{0}e^{-(y-b^{2}/2)t}M_{X_{t}}(n)\dd t&=\frac{(-1)^{n}\sigma^{n}n!}{\sqrt{2^{n}}}\sum^{n}_{k=0}\tilde{c}^{(1)}_{n-k} \sum^{k}_{l=0}\sum^{l}_{i=0}\frac{3^{l}c^{(2)}_{k-l}\Gf(1/2-k+l)}{2^{l}(l-i)!\Gf(1/2-k+i)} \nonumber\\
  &\quad\times\sum^{\infty}_{j=0}\bigg(\frac{b^{2}}{2}\bigg)^{j}\frac{(2j+l)!}{(2j)!}P_{i}(j)y^{-j-1-3n/2}.
\end{align}
Indeed, (\ref{4.16}) only contains negative half-integer powers of $y$. Noting that the infinite series (with respect to $j$) converges absolutely for $\Re y\rightarrow\infty$, by termwise inversion we arrive at the semi-closed formula
\begin{align}\label{4.17}
  M_{X_{t}}(n)&=\frac{(-1)^{n}\sigma^{n}t^{3n/2}n!e^{-\mu^{2}t/(2\sigma^{2})}}{\sqrt{2^{n}}}\sum^{n}_{k=0}\tilde{c}^{(1)}_{n-k} \sum^{k}_{l=0}\sum^{l}_{i=0}\frac{3^{l}c^{(2)}_{k-l}\Gf(1/2-k+l)}{2^{l}\Gf(1/2-k+i)(l-i)!} \nonumber\\
  &\quad\;\times\sum^{\infty}_{j=0}\bigg(\frac{\mu^{2}t}{2\sigma^{2}}\bigg)^{j}\frac{(2j+l)!}{(2j)!\Gf(j+3n/2+1)}P_{i}(j),
\end{align}
where, again, recall that $\mu/\sigma=b$ and the series coefficients are given by (\ref{4.4}) and (\ref{4.15}).

Observe that the $j$-series in (\ref{4.17}) is the special value at $p=1$ of the one-sided Z-transform of $(\mu^{2}t/(2\sigma^{2}))^{j}(2j+l)!/((2j)!\Gf(j+3n/2+1))P_{i}(j)$ as a function of $j\in\mathds{N}$; i.e., we can define
\begin{equation*}
  S_{i}(p)\equiv S_{i}\bigg(p\bigg|\frac{\mu^{2}t}{2\sigma^{2}}\bigg) :=\sum^{\infty}_{j=0}\bigg(\frac{\mu^{2}t}{2\sigma^{2}p}\bigg)^{j}\frac{(2j+l)!}{(2j)!\Gf(j+3n/2+1)}P_{i}(j),\quad\Re p>0.
\end{equation*}
Then, using the differentiation property of the Z-transform, if we apply the Z-transform to the recurrence relation (\ref{4.15}) multiplied by $(\mu^{2}t/(2\sigma^{2}))^{j}(2j+l)!/((2j)!\Gf(j+3n/2+1))$ (independent of $i$), we easily obtain the transformed relation:
\begin{align}\label{4.18}
  &S_{0}(p)=\sum^{\infty}_{j=0}\bigg(\frac{\mu^{2}t}{2\sigma^{2}p}\bigg)^{j}\frac{(2j+l)!}{(2j)!\Gf(j+3n/2+1)} \nonumber\\
  \rightsquigarrow&S_{i}(p)=\frac{\Gf(2/3)}{i}\sum^{i}_{m=1}\frac{(m-i)S_{i-m}(p)-2mpS'_{i-m}(p)}{\Gf(2/3-m)(m+1)!},
\end{align}
where the first line is the same as that of (\ref{4.5}) from standard manipulations with the definition of generalized hypergeometric functions (\cite[Slater, 1966]{S2}) and the derivative $S'_{i}(p)$ in the second line is with respect to $p$. As a differential--difference equation, (\ref{4.5}) is well-posed and has a unique solution simply because it is purely forward in $i\in\mathds{N}$, or is of retarded type.

Therefore, we complete the proof after rewriting (\ref{4.17}) with (\ref{4.5}) and shortening terms.
\end{proof}

Application of Theorem \ref{thm:4} is straightforward: All the series coefficients in (\ref{4.4}) may be computed a priori up to arbitrary orders and their values stored; then, the equation (\ref{4.5}) can be solved explicitly by forward iteration, using the simple fact that continued differentiation only changes the parameters and does not alter the structure (dimensionality) of $\;_{2}\F_{2}$ (\cite[Slater, 1966]{S2}). In this way one can at least obtain (crude) moment formulae in terms of $\;_{2}\F_{2}$.

On another look, $\;_{2}\F_{2}$ in (\ref{4.5}) only has integer or half-integer parameters and so is always reducible to sums of the exponential and error functions. This means that, when fully simplified, all orders of moments should involve $\erf$ as the only non-elementary function.\footnote{Further inspections indicate that even-order moments can be expressed free of non-elementary functions ($\erf$), which we prefer not to expound in this context.} Achieving such a level of simplicity is completely a matter of collecting numerous terms of combined powers of the parameters $\mu$, $\sigma$, and $t$, which task becomes extremely tedious as soon as $n$ exceeds 2 but fortunately can be assigned to computer algebra. In particular, for $n\in\{0,1,2,3,4\}$, Mathematica$^\circledR$ (\cite[Wolfram Research, Inc., 2023]{W1}) calculates that ($\mu\neq0$)
\begin{align}\label{4.19}
  M_{X_{t}}(0)&=1, \nonumber\\
  M_{X_{t}}(1)&=\frac{\sigma\sqrt{t}(\mu^{2}t-\sigma^{2})}{\sqrt{2\pi}\mu^{2}}e^{-\mu^{2}t/(2\sigma^{2})} +\frac{\mu^{4}t^{2}+\sigma^{4}}{2\mu^{3}}\erf\frac{\mu\sqrt{t}}{\sqrt{2}\sigma}, \nonumber\\
  M_{X_{t}}(2)&=\frac{3\mu^{8}t^{4}+4\mu^{6}\sigma^{2}t^{3}+6\mu^{4}\sigma^{4}t^{2}-36\mu^{2}\sigma^{6}t+96\sigma^{8}}{12\mu^{6}} -\frac{\sigma^{6}(\mu^{2}t+8\sigma^{2})}{\mu^{6}}e^{-\mu^{2}t/(2\sigma^{2})}, \nonumber\\
  M_{X_{t}}(3)&=\frac{\sigma\sqrt{t}(5\mu^{10}t^{5}+15\mu^{8}\sigma^{2}t^{4}+10\mu^{6}\sigma^{4}t^{3} -218\mu^{4}\sigma^{6}t^{2}+1070\mu^{2}\sigma^{8}t-14070\sigma^{10})}{20\sqrt{2\pi}\mu^{8}}e^{-\mu^{2}t/(2\sigma^{2})} \nonumber\\
  &\quad+\frac{\mu^{12}t^{6}+4\mu^{10}\sigma^{2}t^{5}+3\mu^{8}\sigma^{4}t^{4}-32\mu^{6}\sigma^{6}t^{3}+240\mu^{4}\sigma^{8}t^{2} -1152\mu^{2}\sigma^{10}t+2814\sigma^{12}}{8\mu^{9}}\erf\frac{\mu\sqrt{t}}{\sqrt{2}\sigma}, \nonumber\\
  M_{X_{t}}(4)&=\frac{1}{48\mu^{12}}(3\mu^{16}t^{8}+24\mu^{14}\sigma^{2}t^{7}+28\mu^{12}\sigma^{4}t^{6}-168\mu^{10}\sigma^{6}t^{5} +2016\mu^{8}\sigma^{8}t^{4}-18816\mu^{6}\sigma^{10}t^{3} \nonumber\\
  &\qquad+130536\mu^{4}\sigma^{12}t^{2}-607824\mu^{2}\sigma^{14}t+1441440\sigma^{16}) \nonumber\\
  &\quad-\frac{21\sigma^{8}(\mu^{8}t^{4}+32\mu^{6}\sigma^{2}t^{3}+432\mu^{4}\sigma^{4}t^{2}+7168\mu^{2}\sigma^{6}t+91520\sigma^{8})}{64\mu^{12}} e^{-\mu^{2}t/(2\sigma^{2})}.
\end{align}
The zeroth and first moments are the same as what direct computation yields (see (\ref{4.2})).

An alternative method to implement Theorem \ref{thm:4}, as is clear from the proof, is to take advantage of the polynomial recurrence relation (\ref{4.15}) instead of invoking (\ref{4.5}). After calculating the polynomials $P_{i}(j)$ up to some satisfactorily large degree $i$, one proceeds to evaluating the $j$-series under multiplication by $(\mu^{2}t/(2\sigma^{2}))^{j}(2j+l)!/((2j)!\Gf(j+3n/2+1))$. Comprehensibly, this method is likely to be more computationally efficient when the value of $\mu^{2}t/(2\sigma^{2})$ is already known. For the same reason, if emphasis is on numerics, the formula (\ref{4.17}) may have a computational advantage over (\ref{4.3}) in that the $j$-series is rapidly converging due to the gamma function in the denominator and hence can be easily truncated for moderate parameter values.

If we send $\mu\rightarrow0$ in (\ref{4.17}), only the term at $j=0$ will remain in the infinite series and since $P_{i}(0)=\mathds{1}_{\{0\}}(i)$ from (\ref{4.15}) the sum with respect to $i$ disappears as well, hence giving us the much simpler formula
\begin{equation}\label{4.20}
  M_{X_{t}}(n;\mu=0)=\frac{(-1)^{n}\sigma^{n}t^{3n/2}n!}{\sqrt{2^{n}}\Gf(3n/2+1)}\sum^{n}_{k=0}\tilde{c}^{(1)}_{n-k} \sum^{k}_{l=0}\bigg(\frac{3}{2}\bigg)^{l}\frac{\Gf(1/2-k+l)}{\Gf(1/2-k)}c^{(2)}_{k-l},\quad n\in\mathds{N},
\end{equation}
which, without regard to scaling, is equivalent to the one provided in \cite[Tak\'{a}cs, 1993, \text{Thm.} 2]{T1} except that the latter used an additional recurrence relation.

Theorem {\ref{thm:4}} has also given rise to an alternative series representation for the Laplace transform.

\begin{corollary}\label{cor:1}
Assuming $\mu\neq0$, the space Laplace transform (\ref{2.2}) has the alternative series representation
\begin{equation}\label{4.21}
  \bar{f}_{X_{t}}(u)=e^{-\mu^{2}t/(2\sigma^{2})}\sum^{\infty}_{n=0}\frac{(\sigma u)^{n}t^{3n/2}}{\sqrt{2^{n}}}\sum^{n}_{k=0}\tilde{c}^{(1)}_{n-k} \sum^{k}_{l=0}\sum^{l}_{i=0}\binom{l-k-1/2}{l-i}\bigg(\frac{3}{2}\bigg)^{l}c^{(2)}_{k-l}\Sigma\bigg(n,l,i;\frac{\mu^{2}t}{2\sigma^{2}}\bigg), \quad\Re u>0,
\end{equation}
where $\Sigma(n,l,i;\mu^{2}t/(2\sigma^{2}))$ is as specified in (\ref{4.5}).
\end{corollary}

\begin{proof}
This follows immediately from substituting (\ref{4.3}) into (\ref{4.12}).
\end{proof}

Unlike (\ref{2.3}), the representation (\ref{4.21}) is expanded for $u$ near the origin and is more suitable for direct implementation (e.g., characteristic function-based estimation). If $\mu=0$, the specialization of (\ref{4.19}) and (\ref{4.21}) is understood as in the same limit as in (\ref{4.20}).

It is also interesting (and useful) to explore the four crucial statistics, namely the mean, variance, skewness ($\Skew$), and excess kurtosis ($\EKurt$), of $X_{t}$, which has been done for the square integral in \cite[Xia, 2020]{X} as well. Computation is straightforward with (\ref{4.19}) and the relations
\begin{align}\label{4.22}
  \Skew X_{t}&=\frac{M_{X_{t}}(3)-3M_{X_{t}}(2)M_{X_{t}}(1)+2M^{3}_{X_{t}}(1)}{\sqrt{\big(M_{X_{t}}(2)-M^{2}_{X_{t}}(1)\big)^{3}}}, \nonumber\\
  \EKurt X_{t}&=\frac{M_{X_{t}}(4)-4M_{X_{t}}(3)M_{X_{t}}(1)+6M_{X_{t}}(2)M^{2}_{X_{t}}(1)-3M^{4}_{X_{t}}(1)}{\big(M_{X_{t}}(2)-M^{2}_{X_{t}}(1)\big)^{2}}-3.
\end{align}
The drift impact on these two statistics is explained in the next corollary.

\begin{corollary}\label{cor:2}
The following three statements hold. \smallskip\\
(i) The skewness and excess kurtosis of $X_{t}$ both attain their global maxima at $\mu=0$, with the respective maximal values\footnote{The numerical values (up to 10 significant digits) are approximately 1.277368533 and 1.776923532, respectively.}
\begin{equation*}
  \max_{\mu\in\mathds{R}}\Skew X_{t}=\frac{8(4480-1257\pi)}{35\sqrt{(27\pi-64)^{3}}}\quad\text{and} \quad\max_{\mu\in\mathds{R}}\EKurt X_{t}=\frac{3(858112\pi-33453\pi^{2}-2293760)}{280(27\pi-64)^{2}}.
\end{equation*}
(ii) The skewness and excess kurtosis of $X_{t}$ both vanish in the limit as $|\mu|\rightarrow\infty$. \smallskip\\
(iii) The skewness of $X_{t}$ goes to its infimum of 0 in the limit as $|\mu|\rightarrow\infty$, but there exists $\mu^{\ast}>0$ (depending on $\sigma$ and $t$) such that the excess kurtosis attains its global minimum\footnote{Approximately $-0.2621350996$.} at $\mu=\pm\mu^{\ast}$.
\end{corollary}

\begin{proof}
From Theorem \ref{thm:4}, and (\ref{4.22}), we can assume $\sigma=t=1$ without loss of generality.

For statement (i), based on the formulae in (\ref{4.19}), using the expansions of the exponential and error functions around $\mu=0$ (\cite[Abramowitz and Stegun, 1972, \text{Eq.} 7.2.4]{AS}), elementary but tedious calculations show that
\begin{equation*}
  \Skew X_{1}=\frac{8(4480-1257\pi)}{35(27\pi-64)^{3/2}}-\frac{36\pi(195328-61515\pi)}{1925(27\pi-64)^{5/2}}\mu^{2}+O(\mu^{4})\quad\text{as }|\mu|\searrow0,
\end{equation*}
on the right side of which the first coefficient is larger than 1 and the second is negative, and
\begin{equation*}
  \Skew X_{1}\leq
  \begin{cases}
    \displaystyle \frac{8(4480-1257\pi)}{35(27\pi-64)^{3/2}}-\frac{36\pi(195328-61515\pi)}{1925(27\pi-64)^{5/2}}\mu^{2}\quad&\text{if }|\mu|\leq2 \\
    \displaystyle 1,\quad&\text{if }|\mu|>2;
  \end{cases}
\end{equation*}
hence it must be that $\max_{\mu\in\mathds{R}}\Skew X_{1}=8(1256\pi-4480)/(35(27\pi-64)^{3/2})$. A similar argument shows the global maximum of $\EKurt X_{1}$.

Statement (ii) results directly from taking the limit as $|\mu|\rightarrow\infty$ in (\ref{4.19}), using the asymptotic expansions of the error function (\cite[Abramowitz and Stegun, 1972, \text{Eq.} 7.2.14]{AS}).

Towards statement (iii), it is not hard to see from (\ref{4.19}) that the third central moment, namely $M_{X_{1}}(3)-3M_{X_{1}}(2)M_{X_{1}}(1)+2M^{3}_{X_{1}}(1)$, is strictly positive for $\mu\in\mathds{R}$, and then the first conclusion (skewness) follows from statement (ii). For the second conclusion, by statement (ii) and the obvious continuity of $\EKurt X_{1}$ as a function of $\mu$, it suffices to find a point where its value is negative: $\EKurt X_{1}\approx-0.2165605388$ at $\mu=2$.
\end{proof}

Corollary \ref{cor:2} implies that despite being unconditionally right-skewed, the distribution of $X_{t}$ may be platykurtic for moderate values of $\mu$, while it tends to normality as $\mu$ goes to infinity. These properties are notably different from the case of the square integral functional, which is always leptokurtic. Figure \ref{fig:1} further visualizes changes in the skewness and excess kurtosis with respect to the drift parameter.

\begin{figure}[H]
  \centering
  \includegraphics[scale=0.5]{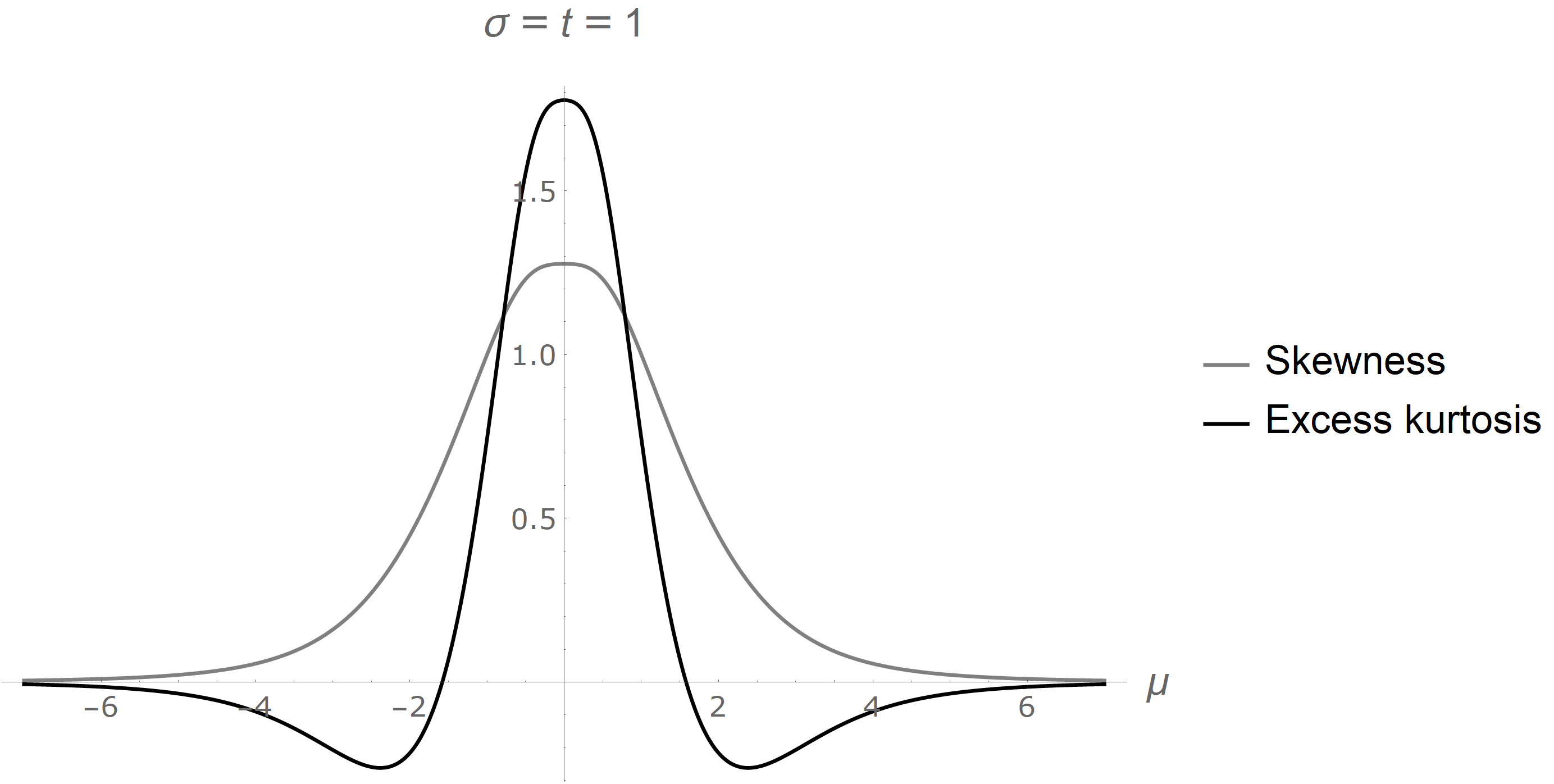}\\
  \caption{Drift impact on skewness and excess kurtosis}
  \label{fig:1}
\end{figure}

\vspace{0.2in}

\section{Asymptotic estimates}\label{sec:5}

In this section we provide asymptotic estimates for the above obtained distribution functions and moments (Theorem \ref{thm:2}, Theorem \ref{thm:3}, and Theorem \ref{thm:4}). We start with the following theorem regarding the small-deviation probabilities of $X_{t}$.

\begin{theorem}\label{thm:5}
The probability density function (\ref{3.1}) and cumulative distribution function (\ref{3.7}) of $X_{t}$ satisfy
\begin{equation}\label{5.1}
  f_{X_{t}}(x)=\frac{\mathcal{I}_{1,0}\sqrt{2\sigma^{2}t^{3}(-\alpha'_{1})}}{3\sqrt{\pi}\Ai(\alpha'_{1})x^{2}} \exp\bigg(-\frac{\mu^{2}t}{2\sigma^{2}}+\frac{2\sigma^{2}(t\alpha'_{1})^{3}}{27x^{2}}\bigg)(1+O(x^{2})),\quad\text{as }x\searrow0,
\end{equation}
and
\begin{equation}\label{5.2}
  F_{X_{t}}(x)=\frac{9\mathcal{I}_{1,0}x}{2\sqrt{2\pi\sigma^{2}t^{3}(-\alpha'_{1})^{5}}\Ai(\alpha'_{1})} \exp\bigg(-\frac{\mu^{2}t}{2\sigma^{2}}+\frac{2\sigma^{2}(t\alpha'_{1})^{3}}{27x^{2}}\bigg)(1+O(x^{2})),\quad\text{as }x\searrow0,
\end{equation}
with the constant\footnote{$\mathcal{I}_{1,0}\approx0.8090732963$.}
\begin{equation}\label{5.3}
  \mathcal{I}_{1,0}=\frac{1}{3}-\frac{\alpha'_{1}}{3^{1/3}} \bigg(\frac{1}{3^{1/3}\Gf(2/3)}{\;_{1}\F_{2}}\bigg(\frac{1}{3};\frac{2}{3},\frac{4}{3}\bigg|\frac{(\alpha'_{1})^{3}}{9}\bigg) +\frac{(-\alpha'_{1})}{2\Gf(1/3)}{\;_{1}\F_{2}}\bigg(\frac{2}{3};\frac{4}{3},\frac{5}{3}\bigg|\frac{(\alpha'_{1})^{3}}{9}\bigg)\bigg).
\end{equation}
\end{theorem}

\begin{proof}
The asymptotic estimate of the G-function of the specific form in (\ref{3.3}) for large arguments is known in \cite[Luke, 1969, \text{Sect.} 5.7 \text{Thm.} 5]{L2}. Together with the relation (\ref{3.4}) it implies that, for any $z>0$ and $j\in\mathds{N}$,
\begin{equation}\label{5.4}
  \G^{0,3}_{3,2}\bigg(
  \begin{array}{ccccc}
    (j+1)/3 & (j+2)/3 & j/3+1 &  &  \\
     &  &  & 1/2 & 1
  \end{array}
  \bigg|z\bigg)=e^{-1/z}z^{j-1/2}(1+O(z)),\quad\text{as }z\searrow0.
\end{equation}
Since the numbers $-\alpha'_{k-j}$, $k\geq0$, are all positive and strictly increasing in $k$, this means that the G-function in (\ref{3.2}) decays like $O(e^{2\sigma^{2}t^{3}(\alpha'_{k-j})^{3}/(27x^{2})})$ as $1/x^{2}\rightarrow\infty$, and so the behavior of $f_{X_{t}}(x)$ for small $x$ is dominated by the leading-order term ($k=1$) in (\ref{3.2}). In particular, letting $S$ be the $k$-series in (\ref{3.2}), we have
\begin{equation*}
  S=\frac{\sqrt{3}\mathcal{I}_{1,0}}{(-\alpha'_{1})\Ai(\alpha'_{1})}\sqrt{\frac{2\sigma^{2}t^{3}(-\alpha'_{1})^{3}}{27x^{2}}} \exp\bigg(\frac{2\sigma^{2}t^{3}(\alpha'_{1})^{3}}{27x^{2}}\bigg)(1+O(x^{2})),\quad\text{as }x\searrow0,
\end{equation*}
where $\mathcal{I}_{1,0}$ is recalled to be given by (\ref{2.8}). Putting this back in (\ref{3.2}) and simplifying we obtain (\ref{5.1}). Following the same routine we have (\ref{5.2}), except that the G-function has a different set of parameters.
\end{proof}

On the other hand, large-deviation probabilities are more interesting because Theorem \ref{thm:2} and Theorem \ref{thm:3}, which come from inverting the asymptotic expansion (\ref{2.3}) of the Laplace transform, are not computationally sustainable when $x$ becomes large. For this purpose we conduct a separate analysis using the space--time Laplace transform in (\ref{4.7}), based on extending the method of \cite[Tolmatz, 2003]{T4}. Before that, the next few lemmata will be useful.

\begin{lemma}\label{lem:2}
With $b=\mu/\sigma\in\mathds{R}$ and the space--time Laplace transform in (\ref{4.7}), for any fixed $y$ with $\Re y>b^{2}/2$,
\begin{equation*}
  \bigg|\phi(u,y)-\frac{1}{y-b^{2}/2}\bigg|\leq C_{y}|u|,
\end{equation*}
where $C_{y}$ is a positive constant depending on $y$.
\end{lemma}

\begin{proof}
The fact that for $\Re u>0$, $|e^{-uX_{t}}-1|\leq|u|X_{t}$, $\PP$-a.s., along with Jensen's inequality gives
\begin{equation*}
  |\bar{f}_{X_{t}}(u)-1|\leq|u|\E X_{t}.
\end{equation*}
Thus, for $\Re y>b^{2}/2$,
\begin{align*}
  \bigg|\phi(u,y)-\frac{1}{y-b^{2}/2}\bigg|&\leq\bigg|\int^{\infty}_{0}e^{-(y-b^{2}/2)t}(\bar{f}_{X_{t}}(u)-1)\dd t\bigg| \\
  &\leq\int^{\infty}_{0}e^{-(\Re y-b^{2}/2)t}|\bar{f}_{X_{t}}(u)-1|\dd t \\
  &\leq|u|\int^{\infty}_{0}e^{-(\Re y-b^{2}/2)t}\E X_{t}\dd t.
\end{align*}
Based on (\ref{4.2}), there exists a constant $C>0$ such that $\E X_{t}\leq C\max\{t^{2},1\}$, which with the last inequality proves the claim.
\end{proof}

\begin{lemma}\label{lem:3}
Let $w_{1},w_{2},z_{1},z_{2}\in\mathds{C}$ and $\Bi\equiv\Bi(\cdot)$ denote the second Airy function. Then it holds that
\begin{equation*}
  \frac{e^{2\pi\ii/3}w_{1}}{\Ai'(e^{2\pi\ii/3}z_{1})}-\frac{e^{-2\pi\ii/3}w_{2}}{\Ai'(e^{-2\pi\ii/3}z_{1})}=\frac{2((w_{2}-w_{1})\Ai'(z_{1}) -\ii(w_{1}+w_{2})\Bi'(z_{1}))}{\Ai'^{2}(z_{1})+\Bi'^{2}(z_{1})}
\end{equation*}
and
\begin{equation*}
  \frac{e^{\pi\ii/3}\Ai(e^{2\pi\ii/3}z_{2})}{\Ai'(e^{2\pi\ii/3}z_{1})}-\frac{e^{-\pi\ii/3}\Ai(e^{-2\pi\ii/3}z_{2})}{\Ai'(e^{-2\pi\ii/3}z_{1})} =\frac{2\ii(\Ai'(z_{1})\Bi(z_{2})-\Ai(z_{2})\Bi'(z_{1}))}{\Ai'^{2}(z_{1})+\Bi'^{2}(z_{1})}.
\end{equation*}
\end{lemma}

\begin{proof}
These formulae are direct consequences of the well-known connection formulae for the Airy functions (see, e.g., \cite[Valle\'{e} and Soares, 2010, \text{Eqs.} 2.15 and 2.16]{VS}),
\begin{equation*}
  \Ai(e^{\pm2\pi\ii/3}z_{2})=\frac{e^{\pm\ii\pi/3}}{2}(\Ai(z_{2})\mp\ii\Bi(z_{2}))\quad\text{and}\quad \Ai'(e^{\pm2\pi\ii/3}z_{1})=\frac{e^{\pm\ii\pi/3}}{2}(\Ai'(z_{1})\mp\ii\Bi'(z_{1})).
\end{equation*}
\end{proof}

\begin{lemma}\label{lem:4}
Let $\bar{\Ai}(p):=\int^{\infty}_{0}e^{-pz}\Ai(z)\dd z$, $p\in\mathds{C}$, denote the (extended) Laplace transform of the first Airy function. Then it holds that
\begin{equation*}
  \bar{\Ai}(pe^{-2\pi\ii/3})+\bar{\Ai}(p)+\bar{\Ai}(pe^{2\pi\ii/3})=e^{-p^{3}/3}.
\end{equation*}
\end{lemma}

\begin{proof}
Based on the explicit formula (\cite[Prudnikov et al., 1992a, \text{Eq.} 30.2.1.1]{PBMa})
\begin{equation}\label{5.5}
  \bar{\Ai}(p)=e^{-p^{3}/3}\bigg(\frac{1}{3}-\frac{p}{3^{4/3}\Gf(4/3)}\;_{1}\F_{1}\bigg(\frac{1}{3};\frac{4}{3}\bigg|\frac{p^{3}}{3}\bigg) +\frac{p^{2}}{3^{5/3}\Gf(5/3)}\;_{1}\F_{1}\bigg(\frac{2}{3};\frac{5}{3}\bigg|\frac{p^{3}}{3}\bigg)\bigg),
\end{equation}
observe that multiplication of $p$ by $e^{\pm2\pi\ii/3}$ does not change the exponential or $\;_{1}\F_{1}$, while $e^{-2\pi\ii/3}+e^{2\pi\ii/3}+1=0$.
\end{proof}

The following lemma can be seen as an adequate generalization of the well-known method of Laplace for approximating integrals (see, e.g., \cite[Wong, 2001, \text{Sect.} 9]{W2}) when the exponent involves more than one large parameter. To our knowledge, integrals of this kind were first analyzed in \cite[Fulks, 1951]{F}, but there has been no version in the multivariate complex-valued case showing relative error orders.

\begin{lemma}[\underline{A generalized Laplace's method}]\label{lem:5}
Let $g_{1},g_{2},g_{3}$ be complex-valued functions defined on a (possibly unbounded) domain $D\subset\mathds{R}^{d}$, $d\geq1$, such that $g_{1}\in\mathcal{C}^{1}(D)$, $g_{2}\in\mathcal{C}^{2}(D)$, and $g_{3}\in\mathcal{C}^{3}(D)$. Suppose that \smallskip\\
(i) $\Re g_{3}$ attains its global minimum at a unique interior point $w_{0}$ of $D$, at which $\nabla g_{3}(w_{0})=\mathbf{0}$, $\Re\mathrm{Hess}_{g_{3}}(w_{0})$ is positive-definite, and $g_{1}(w_{0})\neq0$; \smallskip\\
(ii) there exists a constant $c_{1}>0$ such that the function $g_{1}e^{-c_{1}g_{3}}$ is absolutely integrable; \smallskip\\
(iii) there exists a constant $c_{2}>0$ such that $\Re g_{2}\leq c_{2}\Re g_{3}$. \smallskip \\
Under these three conditions, it holds that in the limit as $x\rightarrow\infty$,
\begin{align}\label{5.6}
  \int_{D}g_{1}(w)e^{xg_{2}(w)-x^{2}g_{3}(w)}\dd w&=\frac{(2\pi)^{d/2}g_{1}(w_{0})}{x^{d}(\det\mathrm{Hess}_{g_{3}}(w_{0}))^{1/2}}\exp\bigg(xg_{2}(w_{0})-x^{2}g_{3}(w_{0}) \nonumber\\
  &\qquad+\frac{1}{2}(\nabla g_{2})^{\intercal}(w_{0})\mathrm{Hess}^{-1}_{g_{3}}(w_{0})\nabla g_{2}(w_{0})\bigg)(1+O(x^{-1})).
\end{align}
\end{lemma}

\begin{proof}
Under condition (i), let $\epsilon>0$ be such that the open ball with radius $\epsilon$ centered at $w_{0}$, $B_{\epsilon}(w_{0})\subset D$. Consider the decomposition $\int_{D}=\int_{D\setminus B_{\epsilon}(w_{0})}+\int_{B_{\epsilon}(w_{0})}$. For the integral
\begin{equation*}
  J_{0}(x)=\int_{D\setminus B_{\epsilon}(w_{0})}g_{1}(w)e^{xg_{2}(w)-x^{2}g_{3}(w)}\dd w,\quad x>0,
\end{equation*}
condition (iii) implies that for some $\delta>0$,
\begin{align*}
  |J_{0}(x)|&\leq\int_{D\setminus B_{\epsilon}(w_{0})}|g_{1}(w)|e^{x\Re g_{2}(w)-x^{2}\Re g_{3}(w)}\dd w \\
  &\leq\int_{D\setminus B_{\epsilon}(w_{0})}|g_{1}(w)|e^{(c_{2}x-x^{2})\Re g_{3}(w)}\dd w \\
  &=e^{(c_{2}x-x^{2})(\Re g_{3}(w_{0})+\delta)}\int_{D\setminus B_{\epsilon}(w_{0})}|g_{1}(w)|e^{(c_{2}x-x^{2})(\Re g_{3}(w)-\Re g_{3}(w_{0})-\delta)}\dd w.
\end{align*}
For sufficiently large $x$, condition (ii) then implies
\begin{equation}\label{5.7}
  |J_{0}(x)|\leq e^{(c_{2}x-x^{2})(\Re g_{3}(w_{0})+\delta)}\int_{D\setminus B_{\epsilon}(w_{0})}|g_{1}(w)|e^{-c_{1}(\Re g_{3}(w)-\Re g_{3}(w_{0})-\delta)}\dd w,
\end{equation}
with the integral on the right side being finite.

Note that by Taylor's theorem there exist $\varpi_{1}\equiv\varpi_{1}(w)$ and $\varpi_{2}\equiv\varpi_{2}(w)$ within $B_{\epsilon}(w_{0})$ and functions $G_{a}:\mathds{R}^{d}\mapsto\mathds{C}$, $|a|=3$, such that\footnote{With the multi-index notation, the sum in the third equation is taken over all $a\in\mathds{N}^{d}$ such that $|a|:=\sum^{d}_{j=1}a_{j}=3$; also, $\mathrm{D}^{a}=\pd^{|a|}/(\prod^{d}_{j}\pd w^{a_{j}}_{j})$ and $(w-w_{0})^{a}=\prod^{d}_{j=1}(w_{j}-w_{0,j})^{a_{j}}$.}
\begin{align*}
  g_{1}(w)&=g_{1}(w_{0})+(\nabla g_{1})^{\intercal}(\varpi_{1})(w-w_{0}), \\
  g_{2}(w)&=g_{2}(w_{0})+(\nabla g_{2})^{\intercal}(w_{0})(w-w_{0})+\frac{1}{2}(w-w_{0})^{\intercal}\mathrm{Hess}_{g_{2}}(\varpi_{2})(w-w_{0}), \\
  g_{3}(w)&=g_{3}(w_{0})+\frac{1}{2}(w-w_{0})^{\intercal}\mathrm{Hess}_{g_{3}}(w_{0})(w-w_{0})+\sum_{|a|=3}\bigg(\frac{\mathrm{D}^{a}g_{3}(w_{0})}{a!} +G_{a}(w)\bigg)(w-w_{0})^{a},
\end{align*}
and $\lim_{w\rightarrow w_{0}}G_{a}(w)=0$. Substituting these into the integral $\int_{B_{\epsilon}(w_{0})}$ and rearranging terms, we obtain
\begin{equation*}
  \int_{B_{\epsilon}(w_{0})}g_{1}(w)e^{x(g_{2}(w)-g_{2}(w_{0}))-x^{2}(g_{3}(w)-g_{3}(w_{0}))}\dd w=J_{1}(x)+J_{2}(x),
\end{equation*}
where
\begin{align*}
  J_{1}(x)&=g_{1}(w_{0})\int_{B_{\epsilon}(0)}\exp\bigg(x(\nabla g_{2})^{\intercal}(w_{0})w-\frac{x^{2}}{2}w^{\intercal}\mathrm{Hess}_{g_{3}}(w_{0})w+\frac{x}{2}w^{\intercal}\mathrm{Hess}_{g_{2}}(\varpi_{2})w \\
  &\qquad-x^{2}\sum_{|a|=3} \bigg(\frac{\mathrm{D}^{a}g_{3}(w_{0})}{a!}+G_{a}(w+w_{0})\bigg)w^{a}\bigg)\dd w, \\
  J_{2}(x)&=\int_{B_{\epsilon}(0)}(\nabla g_{1})^{\intercal}(\varpi_{1})w\exp\bigg(x(\nabla g_{2})^{\intercal}(w_{0})w-\frac{x^{2}}{2}w^{\intercal}\mathrm{Hess}_{g_{3}}(w_{0})w+\frac{x}{2}w^{\intercal}\mathrm{Hess}_{g_{2}}(\varpi_{2})w \\
  &\qquad-x^{2}\sum_{|a|=3} \bigg(\frac{\mathrm{D}^{a}g_{3}(w_{0})}{a!}+G_{a}(w+w_{0})\bigg)w^{a}\bigg)\dd w.
\end{align*}

Let us consider $J_{1}(x)$. Further applying the substitution $w\mapsto w/x$, we have that
\begin{align*}
  J_{1}(x)&=\frac{g_{1}(w_{0})}{x^{d}}\int_{B_{\epsilon x}(0)}\exp\bigg((\nabla g_{2})^{\intercal}(\varpi_{2})w-\frac{1}{2}w^{\intercal}\mathrm{Hess}_{g_{3}}(w_{0})w+\frac{1}{x}\bigg(\frac{1}{2} w^{\intercal}\mathrm{Hess}_{g_{2}}(\varpi_{2})w \\
  &\qquad-\sum_{|a|=3}\bigg(\frac{\mathrm{D}^{a}g_{3}(w_{0})}{a!}+G_{a}\bigg(\frac{w}{x}+w_{0}\bigg)\bigg)w^{a}\bigg)\bigg)\dd w \\
  &=\frac{g_{1}(w_{0})}{x^{d}}\int_{B_{\epsilon x}(0)}e^{(\nabla g_{2})^{\intercal}(\varpi_{2})w-\frac{1}{2}w^{\intercal}\mathrm{Hess}_{g_{3}}(w_{0})w} \\
  &\qquad\times\Bigg(1+\sum^{\infty}_{m=1}\frac{1}{x^{m}m!}\bigg(\frac{1}{2}w^{\intercal}\mathrm{Hess}_{g_{2}}(\varpi_{2})w -\sum_{|a|=3}\bigg(\frac{\mathrm{D}^{a}g_{3}(w_{0})}{a!}+G_{a}\bigg(\frac{w}{x}+w_{0}\bigg)\bigg)w^{a}\bigg)^{m}\Bigg)\dd w.
\end{align*}
Since $\Re\mathrm{Hess}_{g_{3}}(w_{0})$ is positive-definite (condition (i)), the Fubini theorem is applicable for interchanging integration and summation, and we can write
\begin{equation*}
  J_{1}(x)=\sum^{\infty}_{m=0}J_{1,m}(x),
\end{equation*}
where the series converges absolutely for $x\rightarrow\infty$; in particular,
\begin{align}\label{5.8}
  x^{d}J_{1,0}(x)&=g_{1}(w_{0})\int_{B_{\epsilon x}(0)}e^{(\nabla g_{2})^{\intercal}(w_{0})w-w^{\intercal}\mathrm{Hess}_{g_{3}}(w_{0})w/2}\dd w \nonumber\\
  &=(2\pi)^{d/2}g_{1}(w_{0})\frac{\exp((\nabla g_{2})^{\intercal}(w_{0})\mathrm{Hess}^{-1}_{g_{3}}(w_{0})\nabla g_{2}(w_{0})/2)}{(\det\mathrm{Hess}_{g_{3}}(w_{0}))^{1/2}}(1+o(e^{-x})),\quad\text{as }x\rightarrow\infty,
\end{align}
where the second equality follows directly from the property of the multivariate normal distribution. Similarly, for $m\geq1$,
\begin{align}\label{5.9}
  x^{d+m}J_{1,m}(x)&=\frac{g_{1}(w_{0})}{m!}\int_{B_{\epsilon x}(0)}e^{(\nabla g_{2})^{\intercal}(w_{0})w-w^{\intercal}\mathrm{Hess}_{g_{3}}(w_{0})w/2} \nonumber\\
  &\qquad\times\bigg(\frac{1}{2}w^{\intercal}\mathrm{Hess}_{g_{2}}(\varpi_{2})w -\sum_{|a|=3}\bigg(\frac{\mathrm{D}^{a}g_{3}(w_{0})}{a!}+G_{a}\bigg(\frac{w}{x}+w_{0}\bigg)\bigg)w^{a}\bigg)^{m}\dd w,
\end{align}
which also converges to a finite value as $x\rightarrow\infty$, because $\mathrm{Hess}_{g_{2}}(\varpi_{2})$ and $G_{a}(w/x+w_{0})$ are both bounded for $w\in B_{\epsilon x}(0)$ by the continuity of $g_{2}$ and $g_{3}$ and the multivariate normal distribution has finite moments of all orders. Combining (\ref{5.8}) and (\ref{5.9}), we have
\begin{equation*}
  J_{1}(x)=\frac{(2\pi)^{d/2}g_{1}(w_{0})}{x^{d}}\frac{\exp((\nabla g_{2})^{\intercal}(w_{0})\mathrm{Hess}^{-1}_{g_{3}}(w_{0})\nabla g_{2}(w_{0})/2)}{(\det\mathrm{Hess}_{g_{3}}(w_{0}))^{1/2}}(1+O(x^{-1})),\quad\text{as }x\rightarrow\infty.
\end{equation*}
In the same vein, one can show that $J_{2}(x)=O(x^{-(d+1)})$ as $x\rightarrow\infty$. Hence,
\begin{equation*}
  J_{1}(x)+J_{2}(x)=\frac{(2\pi)^{d/2}g_{1}(w_{0})}{x^{d}}\frac{\exp((\nabla g_{2})^{\intercal}(w_{0})\mathrm{Hess}^{-1}_{g_{3}}(w_{0})\nabla g_{2}(w_{0})/2)}{(\det\mathrm{Hess}_{g_{3}}(w_{0}))^{1/2}}(1+O(x^{-1})),\quad\text{as }x\rightarrow\infty.
\end{equation*}

To conclude, note that the desired integral
\begin{equation*}
  \int_{D}g_{1}(w)e^{xg_{2}(w)-x^{2}g_{3}(w)}\dd w=J_{0}(x)+e^{xg_{2}(w_{0})-x^{2}g_{3}(w_{0})}(J_{1}(x)+J_{2}(x)),
\end{equation*}
but from (\ref{5.7}), $J_{0}(x)=o(x^{-(d+1)}e^{x\Re g_{2}(w_{0})-x^{2}\Re g_{3}(w_{0})})$ as $x\rightarrow\infty$ because $\delta>0$.
\end{proof}

We are ready to prove the main result regarding the large-deviation probabilities.

\begin{theorem}\label{thm:6}
The probability density function (\ref{3.1}) and cumulative distribution function (\ref{3.7}) of $X_{t}$ satisfy
\begin{equation}\label{5.10}
  f_{X_{t}}(x)=\sqrt{\frac{6}{\pi\sigma^{2}t^{3}}}\exp\bigg(-\frac{3\mu^{2}t}{8\sigma^{2}}-\frac{3x^{2}}{2\sigma^{2}t^{3}}\bigg)\cosh\frac{3\mu x}{2\sigma^{2}t}(1+O(x^{-1})),\quad\text{as }x\rightarrow\infty
\end{equation}
and
\begin{equation}\label{5.11}
  F_{X_{t}}(x)=1-\sqrt{\frac{2\sigma^{2}t^{3}}{3\pi}}\frac{1}{x}\exp\bigg(-\frac{3\mu^{2}t}{8\sigma^{2}}-\frac{3x^{2}}{2\sigma^{2}t^{3}}\bigg) \cosh\frac{3\mu x}{2\sigma^{2}t}(1+O(x^{-1})),\quad\text{as }x\rightarrow\infty.
\end{equation}
\end{theorem}

\begin{proof}
By the inversion formula, the cumulative distribution function can be expressed as the iterated Bromwich integral
\begin{equation}\label{5.12}
  F_{X_{t}}(x)=\frac{1}{(2\pi\ii)^{2}}\int^{c+\ii\infty}_{c-\ii\infty}e^{yt}\dd y\;\int^{d+\ii\infty}_{d-\ii\infty}\frac{e^{ux}}{u}\phi\bigg(u,y+\frac{b^{2}}{2}\bigg)\dd u,
\end{equation}
for arbitrary $c,d>0$. For $v=\sigma u$ as before, define the variable $\xi\equiv\xi(v,y):=2^{1/3}v^{-2/3}(y+b^{2}/2)$. The substitution $2^{1/3}v^{-2/3}z\mapsto z$ in (\ref{4.7}) yields
\begin{align}\label{5.13}
  \phi\bigg(u,y+\frac{b^{2}}{2}\bigg)&=-\frac{1}{(2v)^{2/3}\Ai'(\xi(v,y))}\bigg(e^{-b\xi(v,y)(2v)^{-1/3}} \int^{\infty}_{\xi(v,y)}e^{bz(2v)^{-1/3}}\Ai(z)\dd z \nonumber\\
  &\quad+e^{b\xi(v,y)(2v)^{-1/3}}\int^{\infty}_{\xi(v,y)}e^{-bz(2v)^{-1/3}}\Ai(z)\dd z\bigg) \nonumber\\
  &=-\frac{1}{(2v)^{2/3}\Ai'(\xi(v,y))}\bigg(e^{-b\xi(v,y)(2v)^{-1/3}}\bigg(\bar{\Ai}(-b(2v)^{-1/3})-\int^{\xi(v,y)}_{0}e^{bz(2v)^{-1/3}}\Ai(z)\dd z\bigg) \nonumber\\
  &\quad+e^{b\xi(v,y)(2v)^{-1/3}}\bigg(\bar{\Ai}(b(2v)^{-1/3})-\int^{\xi(v,y)}_{0}e^{-bz(2v)^{-1/3}}\Ai(z)\dd z\bigg)\bigg).
\end{align}

For any fixed $y$ with $\Re y=c$, we observe that $1/\Ai'(\xi(v,y))$ has singularities exactly on the ray $\{\rho e^{\ii(\theta+\pi)}:\rho\geq0\}$ with $\theta=\arg(y+b^{2}/2)$, so $\phi(u,y+b^{2}/2)$ can be analytically continued into the complex $u$-plane with a cut along this ray, denoted as $D_{\theta}$. The principal values of subsequent fractional powers are understood in $D_{\theta}$.

Observing that any $v\in D_{\theta}$ satisfies $|\arg v^{-1/3}|<\pi/2$, we have that for $|v|\rightarrow\infty$, the Laplace transform $\bar{\Ai}(\pm b(2v)^{-1/3})$ must be bounded (also refer to (\ref{5.5})). With $|e^{\pm\xi(v,y)(2v)^{-1/3}}|$ bounded as well and $|b(2v)^{-1/3}|\rightarrow0$ in the limit, it follows that
\begin{align}\label{5.14}
  \bigg|\phi\bigg(u,y+\frac{b^{2}}{2}\bigg)\bigg|&\leq\frac{1}{|2v|^{2/3}|\Ai'(\xi(v,y))|} \bigg(\big|e^{-b\xi(v,y)(2v)^{-1/3}}\bar{\Ai}(-b(2v)^{-1/3})\big|+\big|e^{b\xi(v,y)(2v)^{-1/3}}\bar{\Ai}(b(2v)^{-1/3})\big| \nonumber\\
  &\qquad+\bigg|e^{-b\xi(v,y)(2v)^{-1/3}}\int^{\xi(v,y)}_{0}e^{bz(2v)^{-1/3}}\Ai(z)\dd z\bigg|+\bigg|e^{b\xi(v,y)(2v)^{-1/3}}\int^{\xi(v,y)}_{0}e^{-bz(2v)^{-1/3}}\Ai(z)\dd z\bigg|\bigg) \nonumber\\
  &\leq C|u|^{-2/3}
\end{align}
for $u$ large and some constant $C>0$ (independent of $y$). This implies that (\ref{5.12}) exists as a double integral and in particular the (inner) $u$-integral is well-defined in $D_{\theta}$, which we focus on next.

For convenience, assume $\sigma=1$ for now, so that $b=\mu$ and $v=u$, and define the regularized function (with infinitesimal $\epsilon>0$)
\begin{equation*}
  G_{\epsilon}(u):=\frac{e^{ux}}{u-\ii\epsilon}\phi\bigg(u,y+\frac{b^{2}}{2}\bigg),\quad u\in D_{\theta},
\end{equation*}
where $\phi$ takes the form in (\ref{5.13}). We will consider the Bromwich integral $1/(2\pi\ii)\int^{d+\ii\infty}_{d-\ii\infty}G_{\epsilon}(u)\dd u$.

This integral can be transformed by using the same type of contour as in \cite[Tolmatz, 2003]{T4}. To be precise, for any fixed $\theta\in(-\pi/2,\pi/2)$, $\delta>0$ small, and $0<r<\epsilon<d<R$, we define a closed contour oriented counterclockwise as follows:
\begin{equation*}
  C_{\delta}=C_{d}\cup C^{\delta}_{R}\cup C^{+}_{\delta}\cup C^{\delta}_{r}\cup C^{-}_{\delta},
\end{equation*}
where
\begin{align*}
  C_{d}&=\{u:\Re u=d,\;|u|\leq R\}, \\
  C^{\delta}_{R}&=\{u:\Re u\leq d,\;|u|=R,\;|\arg u-\theta|<\pi-\delta\}, \\
  C^{\delta}_{r}&=\{u:\Re u\leq d,\;|u|=r,\;|\arg u-\theta|<\pi-\delta\}, \\
  C^{\pm}_{\delta}&=\{u:\arg u=\theta\pm\pi\mp\delta,\;r\leq|u|\leq R\}.
\end{align*}

Since $G_{\epsilon}$ has but one simple pole at $\ii\epsilon$ inside $C_{\delta}$, by Cauchy's residue theorem,
\begin{equation*}
  \int_{C_{\delta}}G_{\epsilon}(u)\dd u=2\pi\ii\underset{u=\ii\epsilon}{\Res}G_{\epsilon}(u).
\end{equation*}
By passing $\delta\searrow0$ we have immediately
\begin{equation*}
  \bigg(\int_{C_{d}}+\int_{C^{0}_{R}}+\int_{C^{+}_{0}}+\int_{C^{0}_{r}}+\int_{C^{-}_{0}}\bigg)G_{\epsilon}(u)\dd u=2\pi\ii\underset{u=\ii\epsilon}{\Res}G_{\epsilon}(u).
\end{equation*}

By (\ref{5.14}) we know $|G_{\epsilon}(u)|\leq Ce^{dx}R^{-5/3}$ for $u\in C^{0}_{R}$, so the integral $\int_{C^{0}_{R}}$ vanishes in the limit as $R\rightarrow\infty$. In the same limit the integral $\int_{C_{d}}$ becomes the Bromwich integral of $G_{\epsilon}$. For the integral $\int_{C^{0}_{r}}$, by Lemma \ref{lem:2} and analytic continuation we know that $\phi(u,y+b^{2}/2)=1/y+O(|u|)$ for $|u|\rightarrow0$ and so $|G_{\epsilon}(u)|$ is bounded for $u\in C^{0}_{r}$; thus the integral $\int_{C^{0}_{r}}$ also vanishes in the limit as $r\searrow0$.

For the two integrals along the ray, by an obvious re-parametrization we have
\begin{equation}\label{5.15}
  \int_{C^{+}_{0}}+\int_{C^{-}_{0}}=\int^{\infty}_{0}\frac{e^{\ii\theta}e^{-\rho e^{\ii\theta}x}e^{-2\ii\theta/3}}{(2\rho)^{2/3}(\rho e^{\ii\theta}+\ii\epsilon)}\big(e^{-2\pi\ii/3}\psi_{+}(\rho,\theta)-e^{2\pi\ii/3}\psi_{-}(\rho,\theta)\big)\dd\rho,
\end{equation}
where
\begin{align}\label{5.16}
  \psi_{\pm}(\rho,\theta)&=\frac{1}{\Ai'(e^{\mp2\pi\ii/3}\xi(\rho))}\bigg(e^{-be^{\mp2\pi\ii/3}\xi(\rho)(2\rho e^{\ii(\theta\pm\pi)})^{-1/3}}\int^{\infty}_{e^{\mp2\pi\ii/3}\xi(\rho)}e^{bz(2\rho e^{\ii(\theta\pm\pi)})^{-1/3}}\Ai(z)\dd z \nonumber\\
  &\quad+e^{be^{\mp2\pi\ii/3}\xi(\rho)(2\rho e^{\ii(\theta\pm\pi)})^{-1/3}}\int^{\infty}_{e^{\mp2\pi\ii/3}\xi(\rho)}e^{-bz(2\rho e^{\ii(\theta\pm\pi)})^{-1/3}}\Ai(z)\dd z\bigg) \nonumber\\
  &=\frac{1}{\Ai'(e^{\mp2\pi\ii/3}\xi(\rho))}\bigg(e^{-be^{\mp2\pi\ii/3}\xi(\rho)(2\rho e^{\ii(\theta\pm\pi)})^{-1/3}}\bar{\Ai}(-b(2\rho e^{\ii(\theta\pm\pi)})^{-1/3}) \nonumber\\
  &\qquad+e^{be^{\mp2\pi\ii/3}\xi(\rho)(2\rho e^{\ii(\theta\pm\pi)})^{-1/3}}\bar{\Ai}(b(2\rho e^{\ii(\theta\pm\pi)})^{-1/3}) \nonumber\\
  &\quad-\bigg(e^{-be^{\mp2\pi\ii/3}\xi(\rho)(2\rho e^{\ii(\theta\pm\pi)})^{-1/3}}\int^{e^{\mp2\pi\ii/3}\xi(\rho)}_{0}e^{bz(2\rho e^{\ii(\theta\pm\pi)})^{-1/3}}\Ai(z)\dd z \nonumber\\
  &\qquad+e^{be^{\mp2\pi\ii/3}\xi(\rho)(2\rho e^{\ii(\theta\pm\pi)})^{-1/3}}\int^{e^{\mp2\pi\ii/3}\xi(\rho)}_{0}e^{-bz(2\rho e^{\ii(\theta\pm\pi)})^{-1/3}}\Ai(z)\dd z\bigg)\bigg),
\end{align}
with $\xi\equiv\xi(\rho)=2^{1/3}\rho^{-2/3}e^{-2\ii\theta/3}(y+b^{2}/2)$ re-parameterized as well. After applying the substitutions $z\mapsto e^{\mp2\pi\ii/3}z$ to the integrals in the second equality and rearranging terms in (\ref{5.15}) and (\ref{5.16}), we have
\begin{equation}\label{5.17}
  \int_{C^{+}_{0}}+\int_{C^{-}_{0}}=-\int^{\infty}_{0}\frac{e^{\ii\theta}e^{-2\ii\theta/3}e^{-\rho e^{\ii\theta}x}}{(2\rho)^{2/3}(\rho e^{\ii\theta}+\ii\epsilon)}\bigg(A_{1}(\rho)+\int^{\xi(\rho)}_{0}A_{2}(\rho,z)\dd z\bigg)\dd\rho,
\end{equation}
where
\begin{align*}
  A_{1}(\rho)&=\frac{e^{2\pi\ii/3}}{\Ai'(e^{2\pi\ii/3}\xi(\rho))}\big(e^{-be^{2\pi\ii/3}\xi(\rho)(2\rho e^{\ii(\theta-\pi)})^{-1/3}}\bar{\Ai}(-b(2\rho e^{\ii(\theta-\pi)})^{-1/3}) \\
  &\qquad+e^{be^{2\pi\ii/3}\xi(\rho)(2\rho e^{\ii(\theta-\pi)})^{-1/3}}\bar{\Ai}(b(2\rho e^{\ii(\theta-\pi)})^{-1/3})\big) \\
  &\quad-\frac{e^{-2\pi\ii/3}}{\Ai'(e^{-2\pi\ii/3}\xi(\rho))}\big(e^{-be^{-2\pi\ii/3}\xi(\rho)(2\rho e^{\ii(\theta+\pi)})^{-1/3}}\bar{\Ai}(-b(2\rho e^{\ii(\theta+\pi)})^{-1/3}) \\
  &\qquad+e^{be^{-2\pi\ii/3}\xi(\rho)(2\rho e^{\ii(\theta+\pi)})^{-1/3}}\bar{\Ai}(b(2\rho e^{\ii(\theta+\pi)})^{-1/3})\big) \\
  &=e^{b\xi(\rho)(2\rho e^{\ii\theta})^{-1/3}}\bigg(\frac{e^{2\pi\ii/3}\bar{\Ai}(-b(2\rho e^{\ii(\theta-\pi)})^{-1/3})}{\Ai'(e^{2\pi\ii/3}\xi(\rho))}-\frac{e^{-2\pi\ii/3}\bar{\Ai}(-b(2\rho e^{\ii(\theta+\pi)})^{-1/3})}{\Ai'(e^{-2\pi\ii/3}\xi(\rho))}\bigg) \\
  &\quad+e^{-b\xi(\rho)(2\rho e^{\ii\theta})^{-1/3}}\bigg(\frac{e^{2\pi\ii/3}\bar{\Ai}(b(2\rho e^{\ii(\theta-\pi)})^{-1/3})}{\Ai'(e^{2\pi\ii/3}\xi(\rho))}-\frac{e^{-2\pi\ii/3}\bar{\Ai}(b(2\rho e^{\ii(\theta+\pi)})^{-1/3})}{\Ai'(e^{-2\pi\ii/3}\xi(\rho))}\bigg)
\end{align*}
and similarly,
\begin{align*}
  A_{2}(\rho,z)&=e^{\pi\ii/3}\big(e^{b(2\rho e^{\ii\theta})^{-1/3}(\xi(\rho)-z)}+e^{-b(2\rho e^{\ii\theta})^{-1/3}(\xi(\rho)-z)}\big)\frac{\Ai(e^{2\pi\ii/3}z)}{\Ai'(e^{2\pi\ii/3}\xi(\rho))} \\
  &\quad-e^{-\pi\ii/3}\big(e^{b(2\rho e^{\ii\theta})^{-1/3}(\xi(\rho)-z)}+e^{-b(2\rho e^{\ii\theta})^{-1/3}(\xi(\rho)-z)}\big)\frac{\Ai(e^{-2\pi\ii/3}z)}{\Ai'(e^{-2\pi\ii/3}\xi(\rho))}.
\end{align*}
Then, by applying Lemma \ref{lem:3} we have the following equivalent expressions:
\begin{align}\label{5.18}
  A_{1}(\rho)&=\frac{2}{\Ai'^{2}(\xi(\rho))+\Bi'^{2}(\xi(\rho))}\bigg(e^{b\xi(\rho)(2\rho e^{\ii\theta})^{-1/3}}\big((\bar{\Ai}(-b(2\rho e^{\ii(\theta+\pi)})^{-1/3})-\bar{\Ai}(-b(2\rho e^{\ii(\theta-\pi)})^{-1/3}))\Ai'(\xi(\rho)) \nonumber\\
  &\qquad-\ii(\bar{\Ai}(-b(2\rho e^{\ii(\theta-\pi)})^{-1/3})+\bar{\Ai}(-b(2\rho e^{\ii(\theta+\pi)})^{-1/3}))\Bi'(\xi(\rho))\big) \nonumber\\
  &\quad+e^{-b\xi(\rho)(2\rho e^{\ii\theta})^{-1/3}}\big((\bar{\Ai}(b(2\rho e^{\ii(\theta+\pi)})^{-1/3})-\bar{\Ai}(b(2\rho e^{\ii(\theta-\pi)})^{-1/3}))\Ai'(\xi(\rho)) \nonumber\\
  &\qquad-\ii(\bar{\Ai}(b(2\rho e^{\ii(\theta-\pi)})^{-1/3})+\bar{\Ai}(b(2\rho e^{\ii(\theta+\pi)})^{-1/3}))\Bi'(\xi(\rho))\big)\bigg), \nonumber\\
  A_{2}(\rho,z)&=\frac{2\ii(\Ai'(\xi(\rho))\Bi(z)-\Ai(z)\Bi'(\xi(\rho)))}{\Ai'^{2}(\xi(\rho))+\Bi'^{2}(\xi(\rho))}\big(e^{b(2\rho e^{\ii\theta})^{-1/3}(\xi(\rho)-z)}+e^{-b(2\rho e^{\ii\theta})^{-1/3}(\xi(\rho)-z)}\big).
\end{align}

Now, from Lemma \ref{lem:2} it is immediate that
\begin{equation*}
  \lim_{\epsilon\searrow0}\underset{u=\ii\epsilon}{\Res}G_{\epsilon}(u)=\frac{1}{y},
\end{equation*}
and by putting (\ref{5.18}) back in (\ref{5.17}) and passing $\epsilon\searrow0$ the desired Bromwich integral becomes
\begin{equation}\label{5.19}
  \frac{1}{2\pi\ii}\int^{d+\ii\infty}_{d-\ii\infty}\frac{e^{ux}}{u}\phi\bigg(u,y+\frac{b^{2}}{2}\bigg)\dd u=\frac{1}{y}+\frac{1}{2\pi\ii}\int^{\infty}_{0}\frac{e^{-2\ii\theta/3-\rho e^{\ii\theta}x}}{2^{2/3}\rho^{5/3}}\bigg(A_{1}(\rho)+\int^{\xi(\rho)}_{0}A_{2}(\rho,z)\dd z\bigg)\dd\rho;
\end{equation}
here taking the limit inside the integral is allowed because of (\ref{5.14}). When coupled with (\ref{5.19}), (\ref{5.12}) implies that
\begin{equation*}
  F_{X_{t}}(x;\sigma=1)=1+\frac{1}{2\pi\ii}\int^{c+\ii\infty}_{c-\ii\infty}e^{yt}\dd y\;\frac{1}{2\pi\ii}\int^{\infty}_{0}\frac{e^{-2\ii\theta/3-\rho e^{\ii\theta}x}}{2^{2/3}\rho^{5/3}}\bigg(A_{1}(\rho)+\int^{\xi(\rho)}_{0}A_{2}(\rho,z)\dd z\bigg)\dd\rho,
\end{equation*}
which, for the same reason, exists as a double integral in the domain $\mathds{R}_{+}\times(c+\ii\mathds{R})$.

Recall that $\Re y=c$ and $\Im y=(c+b^{2}/2)\tan\theta$, so we adopt the parametrization $y=c+\ii(c+b^{2}/2)\tan\theta$ to write
\begin{align}\label{5.20}
  F_{X_{t}}(x;\sigma=1)&=1-\frac{c+b^{2}/2}{4\times2^{2/3}\pi^{2}}\iint_{\mathds{R}_{+}\times(-\pi/2,\pi/2)}\rho^{-5/3} \exp\bigg(t\bigg(c+\ii\bigg(c+\frac{b^{2}}{2}\bigg)\tan\theta\bigg)-\frac{2\ii\theta}{3}-\rho e^{\ii\theta}x\bigg)\sec^{2}\theta \nonumber\\
  &\qquad\times\ii\bigg(A_{1}(\rho,\theta)+\int^{\xi(\rho,\theta)}_{0}A_{2}(\rho,\theta,z)\dd z\bigg)\dd(\rho,\theta),
\end{align}
with the arguments of $\xi$, $A_{1}$, and $A_{2}$ expanded to show explicit dependency on $\theta$ (in place of $y$).

Note that $\arg\xi(\rho,\theta)=\theta/3$. By using the asymptotic expansion of $\Bi'$ in \cite[Abramowitz and Stegun, 1972, \text{Eq.} 10.4.66]{AS}, or specifically, $\Bi'(w)=w^{1/4}e^{2/3w^{3/2}}/\sqrt{\pi}(1+O(w^{-3/2}))$ as $|w|\rightarrow\infty$ with $|\arg w|<\pi/3$, and comparing it against (\ref{4.8}) the following estimates are easily justified:
\begin{align}\label{5.21}
  A_{1}(\rho,\theta)&=-\frac{2\ii}{\Bi'(\xi(\rho,\theta))}\big(e^{b\xi(\rho,\theta)(2\rho e^{\ii\theta})^{-1/3}}(\bar{\Ai}(-b(2\rho e^{\ii(\theta-\pi)})^{-1/3})+\bar{\Ai}(-b(2\rho e^{\ii(\theta+\pi)})^{-1/3})) \nonumber\\
  &\quad+e^{-b\xi(\rho,\theta)(2\rho e^{\ii\theta})^{-1/3}}(\bar{\Ai}(b(2\rho e^{\ii(\theta-\pi)})^{-1/3})+\bar{\Ai}(b(2\rho e^{\ii(\theta+\pi)})^{-1/3}))\big)(1+E_{1}(\rho,\theta)),
\end{align}
with the remainder estimate $|E_{1}(\rho,\theta)|<C_{1}e^{-\gamma_{1}|\xi(\rho,\theta)|^{3/2}}$ for some positive constants $C_{1}$ and $\gamma_{1}$. A similar argument shows the existence of positive constants $C_{2}$ and $\gamma_{2}$ such that for large $|\xi(\rho,\theta)|$,
\begin{align}\label{5.22}
  \int^{\xi(\rho,\theta)}_{0}A_{2}(\rho,\theta,z)\dd z&=-\frac{2\ii}{\Bi'(\xi(\rho,\theta))}\bigg(e^{b\xi(\rho,\theta)(2\rho e^{\ii\theta})^{-1/3}}\int^{\xi(\rho,\theta)}_{0}e^{-b(2\rho e^{\ii\theta})^{-1/3}z}\Ai(z)\dd z \nonumber\\
  &\quad+e^{-b\xi(\rho,\theta)(2\rho e^{\ii\theta})^{-1/3}}\int^{\xi(\rho,\theta)}_{0}e^{b(2\rho e^{\ii\theta})^{-1/3}z}\Ai(z)\dd z\bigg)\big(1+E_{2}(\rho,\theta)\big)
\end{align}
and $|E_{2}(\rho,\theta)|<C_{2}e^{-\gamma_{2}|\xi(\rho,\theta)|^{3/2}}$. For the integrals on the right side of (\ref{5.22}), since $\arg\xi(\rho,\theta)=\theta/3$, we have
\begin{align}\label{5.23}
  \bigg|\bigg(\int^{\xi(\rho,\theta)}_{0}-\int^{\infty}_{0}\bigg)e^{\mp b(2\rho e^{\ii\theta})^{-1/3}z}\Ai(z)\dd z\bigg|&=\bigg|\int^{\infty}_{\xi(\rho,\theta)}e^{\mp b(2\rho e^{\ii\theta})^{-1/3}z}\Ai(z)\dd z\bigg| \nonumber\\
  &\leq C_{3}\int^{\infty}_{|\xi(\rho,\theta)|}e^{|b|(2\rho)^{-1/3}z}e^{-2/3\cos(\theta/2)z^{3/2}}\dd z \nonumber\\
  &\leq C_{4}e^{-\gamma_{4}|\xi(\rho,\theta)|^{3/2}}
\end{align}
for positive constants $C_{3}$, $C_{4}$, and $\gamma_{4}$, while $\int^{\infty}_{0}e^{\mp b(2\rho e^{\ii\theta})^{-1/3}z}\Ai(z)\dd z=\bar{\Ai}(\pm b(2\rho e^{\ii\theta})^{-1/3})$. Combining (\ref{5.21}), (\ref{5.22}), and (\ref{5.23}) and merging terms we can claim that, when $\xi$ is treated as an independent variable (with $\arg\xi=\theta/3$),
\begin{equation}\label{5.24}
  \ii\bigg(A_{1}(\rho,\theta)+\int^{\xi}_{0}A_{2}(\rho,\theta,z)\dd z\bigg)=\frac{2H_{b}(\rho,\theta)}{\Bi'(\xi)}(1+E(\rho,\theta)),
\end{equation}
where
\begin{align*}
  H_{b}(\rho,\theta)&:=e^{b\xi(2\rho e^{\ii\theta})^{-1/3}}(\bar{\Ai}(-b(2\rho e^{\ii(\theta-\pi)})^{-1/3})+\bar{\Ai}(b(2\rho e^{\ii\theta})^{-1/3})+\bar{\Ai}(-b(2\rho e^{\ii(\theta+\pi)})^{-1/3})) \\
  &\quad+e^{-b\xi(2\rho e^{\ii\theta})^{-1/3}}(\bar{\Ai}(b(2\rho e^{\ii(\theta-\pi)})^{-1/3})+\bar{\Ai}(-b(2\rho e^{\ii\theta})^{-1/3})+\bar{\Ai}(b(2\rho e^{\ii(\theta+\pi)})^{-1/3}))
\end{align*}
is a drift-dependent function\footnote{It is clear that $H_{-b}$ is the same as $H_{b}$, and note the constancy $H_{0}\equiv2$ in the drift-less case.} and the remainder estimate $|E(\rho,\theta)|<Ce^{-\gamma|\xi|^{3/2}}$ for some positive constants $C$ and $\gamma$.

Since $c>0$ is arbitrary, we shall see that the choices
\begin{equation}\label{5.25}
  c=\frac{9x^{2}}{2t^{4}}\quad\text{and}\quad\rho=\frac{3x}{t^{3}}\varrho,\quad\varrho\geq0,
\end{equation}
as implied from the drift-less case in \cite[Tolmatz, 2003, \text{Eqs.} (92) and (93)]{T4}, are workable for preparing (\ref{5.20}) in a form suitable for applying the generalized Laplace's method in Lemma \ref{lem:5}. Subject to (\ref{5.25}), using the asymptotic expansion of $\Bi'$ again we have that as $x\rightarrow\infty$,
\begin{align}\label{5.26}
  \frac{1}{\Bi'(\xi)}&=\frac{\sqrt{\pi}e^{-2/3\xi^{3/2}}}{\xi^{1/4}}(1+O(\xi^{-3/2})) \nonumber\\
  &=\frac{(3\varrho x)^{1/6}\sqrt{\pi}}{2^{1/12}\sqrt{t}(9x^{2}/(2t^{4})+b^{2}/2)^{1/4}(1+\ii\tan\theta)^{1/4}} \nonumber\\
  &\quad\times\exp\bigg(\frac{\ii\theta}{6}-\frac{2\sqrt{2}t^{3}e^{-\ii\theta}(1+\ii\tan\theta)^{3/2}}{9\varrho x}\bigg(\frac{9x^{2}}{2t^{4}}+\frac{b^{2}}{2}\bigg)^{3/2}\bigg)(1+O(x^{-2})) \nonumber\\
  &=\frac{\sqrt{\pi t}(2\varrho)^{1/6}}{(3x)^{1/3}(1+\ii\tan\theta)^{1/4}} \exp\bigg(\frac{\ii\theta}{6}-\frac{e^{-\ii\theta}(1+\ii\tan\theta)^{3/2}(6x^{2}+b^{2}t^{4})}{2\varrho t^{3}}\bigg)(1+O(x^{-2})).
\end{align}
Also, subject to (\ref{5.25}), by Lemma \ref{lem:4} we have the estimate
\begin{equation}\label{5.27}
  \bar{\Ai}(\mp b(2\rho e^{\ii(\theta-\pi)})^{-1/3})+\bar{\Ai}(\pm b(2\rho e^{\ii\theta})^{-1/3})+\bar{\Ai}(\mp b(2\rho e^{\ii(\theta+\pi)})^{-1/3})=1+O(x^{-1}),
\end{equation}
which implies
\begin{equation}\label{5.28}
  H_{b}(\rho,\theta)=\big(e^{3be^{-\ii\theta}(1+\ii\tan\theta)x/(2\varrho t)}+e^{-3be^{-\ii\theta}(1+\ii\tan\theta)x/(2\varrho t)}\big)(1+O(x^{-1})),\quad\text{as }x\rightarrow\infty.
\end{equation}
After collocating (\ref{5.24}), (\ref{5.26}), and (\ref{5.28}) with (\ref{5.20}) and simplifying things, we obtain for $x$ large,
\begin{align}\label{5.29}
  F_{X_{t}}(x;\sigma=1)&=1-\frac{3x}{4\sqrt{2}(\pi t)^{3/2}}\iint_{\mathds{R}_{+}\times(-\pi/2,\pi/2)}\frac{\sec^{2}\theta}{\varrho^{3/2}(1+\ii\tan\theta)^{1/4}} \nonumber\\
  &\quad\times\big(e^{b\Psi(\varrho,\theta)x/t}+e^{-b\Psi(\varrho,\theta)x/t}\big) \nonumber\\
  &\quad\times\exp\bigg(\frac{\ii b^{2}t\tan\theta}{2}-\frac{b^{2}te^{-\ii\theta}(1+\ii\tan\theta)^{3/2}}{2\varrho}-\frac{\ii\theta}{2} -\frac{x^{2}}{t^{3}}\Phi(\varrho,\theta)\bigg)\dd(\varrho,\theta)\;(1+O(x^{-1})),
\end{align}
with the functions
\begin{align*}
  \Psi(\varrho,\theta)&:=\frac{3e^{-\ii\theta}(1+\ii\tan\theta)}{2\varrho}=\frac{3\sec\theta}{2\varrho}, \\
  \Phi(\varrho,\theta)&:=3\varrho e^{\ii\theta}+3\varrho^{-1}e^{-\ii\theta}(1+\ii\tan\theta)^{3/2}-\frac{9}{2}(1+\ii\tan\theta).
\end{align*}

It is easily verifiable that the function $\Re\Phi(\varrho,\theta)=3\varrho\cos\theta+3\varrho^{-1}\cos(\theta-3/2\arg(1+\ii\tan\theta))\sec^{3/2}\theta-9/2$ attains its global minimum of $3/2$ at $(\varrho,\theta)=(1,0)$, with the Hessian matrix
\begin{equation*}
  \mathrm{Hess}_{\Phi}(1,0)=\left(
  \begin{array}{cc}
    6 & 3\ii/2 \\
    3\ii/2 & 3/4
  \end{array}\right);
\end{equation*}
at the same point, the function $\Psi(\varrho,\theta)$ has the gradient $\nabla\Psi(1,0)=(3/2,0)^{\intercal}$.

These form the basis to apply the generalized Laplace's method in Lemma \ref{lem:5}. Indeed, by specifying $d=2$, $D=\mathds{R}_{+}\times(-\pi/2,\pi/2)$, $g_{2}=\pm b\Psi/t$, $g_{3}=\Phi/t^{3}$, and $g_{1}$ consisting of everything else in the integrand, conditions (i) and (ii) are automatically satisfied, while condition (iii) is verified by checking that $\Psi/\Re\Phi\leq8/7$. Consequently, we obtain
\begin{align}\label{5.30}
  \iint_{\mathds{R}_{++}\times(-\pi/2,\pi/2)}&=\frac{2\pi t^{3}e^{-b^{2}t/2}(e^{3bx/(2t)}+e^{-3bx/(2t)})e^{-3x^{2}/(2t^{3})}} {x^{2}\sqrt{\det\mathrm{Hess}_{\Phi}(1,0)}} \nonumber\\
  &\quad\times\exp\bigg(\frac{b^{2}t(\nabla\Psi)^{\intercal}(1,0)\mathrm{Hess}^{-1}_{\Phi}(1,0)\nabla\Psi(1,0)}{2}\bigg)(1+O(x^{-1})) \nonumber\\
  &=\frac{4\pi t^{3}e^{-3b^{2}t/8}(e^{3bx/(2t)}+e^{-3bx/(2t)})e^{-3x^{2}/(2t^{3})}}{3\sqrt{3}x^{2}}(1+O(x^{-1})).
\end{align}
Putting (\ref{5.30}) into (\ref{5.29}) and simplifying yields
\begin{equation*}
  F_{X_{t}}(x;\sigma=1)=1-\sqrt{\frac{2t^{3}}{3\pi}}\frac{e^{-3b^{2}t/8-3x^{2}/(2t^{3})}}{x}\cosh\frac{3bx}{2t}(1+O(x^{-1})),
\end{equation*}
from which (\ref{5.11}) is recovered upon applying $b=\mu/\sigma$ and $x\mapsto x/\sigma$.

Differentiation of (\ref{5.11}) with respect to $x$ gives the corresponding formula (\ref{5.10}) for the density function after simplifications, completing the proof.
\end{proof}

A few remarks are in order. When rearranged for general $\sigma>0$, (\ref{5.20}) has given an alternative double-integral representation for the cumulative distribution function, namely
\begin{align*}
  F_{X_{t}}(x)&=1-\frac{c+b^{2}/2}{4\times2^{2/3}\pi^{2}}\iint_{\mathds{R}_{+}\times(-\pi/2,\pi/2)}\rho^{-5/3} \exp\bigg(t\bigg(c+\ii\bigg(c+\frac{b^{2}}{2}\bigg)\tan\theta\bigg)-\frac{2\ii\theta}{3}-\frac{\rho e^{\ii\theta}x}{\sigma}\bigg)\sec^{2}\theta \\
  &\qquad\times\ii\bigg(A_{1}(\rho,\theta)+\int^{\xi(\rho,\theta)}_{0}A_{2}(\rho,\theta,z)\dd z\bigg)\dd(\rho,\theta),\quad x>0,
\end{align*}
where $c>0$ is arbitrary, $b=\mu/\sigma$, $\xi(\rho,\theta)=2^{1/3}\rho^{-2/3}e^{-2\ii\theta/3}(c+b^{2}/2)(1+\ii\tan\theta)$, and $A_{1}$ and $A_{2}$ are as specified in (\ref{5.18}).

It is also clear from (\ref{5.10}) that the distribution of $X_{t}$ always has a light right tail, regardless of drift. This observation agrees with the behaviors of its skewness and excess kurtosis discussed in Corollary \ref{cor:2}.

By setting $\mu=0$ in Theorem \ref{thm:6}, the main result of \cite[Tolmatz, 2003]{T4} is immediately recovered. In this case, Lemma \ref{lem:5} becomes the standard Laplace's method for which the relative error order in (\ref{5.6}) is $O(x^{-2})$, and the asymptotic expansion of $\bar{\Ai}$ in (\ref{5.27}) goes away, and thus the associated relative error of $O(x^{-1})$ in (\ref{5.10}) and (\ref{5.11}) can be reduced to $O(x^{-2})$.

Yet another implication from Theorem \ref{thm:6} is the large deviation for the moments. We state the next corollary.

\begin{corollary}\label{cor:3}
The moments (\ref{4.1}) of $X_{t}$ satisfy
\begin{equation}\label{5.31}
  M_{X_{t}}(n)=\bigg(\frac{2}{3}\bigg)^{n/2}\frac{\sigma^{n}t^{3n/2}}{\sqrt{\pi}}\Gf\bigg(\frac{n+1}{2}\bigg) \;_{1}\F_{1}\bigg(-\frac{n}{2};\frac{1}{2}\bigg|-\frac{3\mu^{2}t}{8\sigma^{2}}\bigg)(1+O(n^{-1})),\quad\text{as }n\rightarrow\infty.
\end{equation}
\end{corollary}

\begin{proof}
Using (\ref{5.10}) we have that as $n\rightarrow\infty$,
\begin{align*}
  M_{X_{t}}(n)&=\int^{\infty}_{0}x^{n}f_{X_{t}}(x)\dd x\\
  &=\sqrt{\frac{6}{\pi\sigma^{2}t^{3}}}e^{-3\mu^{2}t/(8\sigma^{2})}\int^{\infty}_{0}x^{n}e^{-3x^{2}/(2\sigma^{2}t^{3})}\cosh\frac{3\mu x}{2\sigma^{2}t}(1+O(x^{-1}))\dd x \\
  &=\sqrt{\frac{3}{2\pi\sigma^{2}t^{3}}}e^{-3\mu^{2}t/(8\sigma^{2})}\int^{\infty}_{0}e^{-3x/(2\sigma^{2}t^{3})}x^{(n-1)/2} \cosh\frac{3\mu\sqrt{x}}{2\sigma^{2}t}(1+O(x^{-1/2}))\dd x.
\end{align*}
The last integral is the Laplace transform of the function $x^{\nu}\cosh(3\mu\sqrt{x})/(2\sigma^{2}t)$ in $x$ evaluated at $p=3/(2\sigma^{2}t^{3})$, where $\nu\in\{(n-1)/2,n/2-1\}$. For example, we may consult the formula in \cite[Prudnikov et al., 1992a, \text{Eq.} 2.2.1.16]{PBMa} and use the connection formulae in \cite[Abramowitz and Stegun, 1972, \text{Eqs.} 19.2.1 and 19.3.1]{AS} to write it in terms of $\;_{1}\F_{1}$ (for notational consistency) and obtain
\begin{equation*}
  M_{X_{t}}(n)=e^{-3\mu^{2}t/(8\sigma^{2})}\bigg(\frac{2}{3}\bigg)^{n/2}\frac{\sigma^{n}t^{3n/2}}{\sqrt{\pi}}\Gf\bigg(\frac{n+1}{2}\bigg) \;_{1}\F_{1}\bigg(\frac{n+1}{2};\frac{1}{2}\bigg|\frac{3\mu^{2}t}{8\sigma^{2}}\bigg)(1+O(n^{-1})),\quad\text{as }n\rightarrow\infty,
\end{equation*}
which is the same as (\ref{5.31}) by the Kummer transformation on $\;_{1}\F_{1}$ (see \cite[Abramowitz and Stegun, 1972, \text{Eqs.} 13.1.27]{AS}).
\end{proof}

\vspace{0.2in}

\section{Numerical experiments}\label{sec:6}

This section serves to illustrate the exact formulae that we have derived before, in particular those in Theorem \ref{thm:2}, Theorem \ref{thm:3}, and Theorem \ref{thm:4}. Similar experiments on the asymptotic formulae given in Theorem \ref{thm:5}, Theorem \ref{thm:6}, and Corollary \ref{cor:3} show that the approximations are generally very sharp in the stated directions for the arguments, but their numerical results are excluded in the paper for a concise presentation.

To truncate the infinite series in (\ref{3.2}) for the probability density function, we may rewrite for any fixed $t>0$ $f_{X_{t}}(x)=\sum^{\infty}_{k=0}h_{k,t}(x)$ with $h_{k,t}(x)$, $(k,x)\in\mathds{N}\times\mathds{R}_{++}$, being the summand in the $k$-series, and consider stopping the series at some (sufficiently large) integer $K>0$. Note that the zeros of $\Ai'$ have the asymptotic estimate (\cite[Abramowitz and Stegun, 1972, \text{Eqs.} 10.4.95 and 10.4.105]{AS})
\begin{equation*}
  \alpha'_{k}=-\bigg(\frac{3\pi(4k-3)}{8}\bigg)^{2/3}(1+O(k^{-2})),\quad\text{as }k\rightarrow\infty.
\end{equation*}
This along with the estimate (\ref{5.4}) implies the following geometric error bound for the remainder:
\begin{equation}\label{6.1}
  \Bigg|\sum^{\infty}_{k=K+1}h_{k,t}(x)\Bigg|<\frac{|h_{K,t}(x)h_{K+1,t}(x)|}{|h_{K,t}(x)|-|h_{K+1,t}(x)|},
\end{equation}
which, of course, increases with the magnitude of drift relative to dispersion, i.e., one will generally demand a larger $K$ for a larger value of $|b|\equiv|\mu|/\sigma$. A similar argument goes for the truncation of the series in (\ref{3.8}) for the cumulative distribution function.

Allowing for easy comparison, in what follows the dispersion parameter and time are fixed at $\sigma=1$ and $t=1$ and only four choices of the drift parameter $\mu\in\{0,1,3/2,2\}$ are considered. This parametric range is sufficient for visualization of the impact of drift on the actual distribution and thus demonstration of overall stability and efficiency of the formulae in applications. Again, recall that changing the sign of $\mu$ makes no difference. It is understood that we will focus on the following functionals:
\begin{equation*}
  \int^{1}_{0}|W_{s}|\dd s,\quad\int^{1}_{0}|s+W_{s}|\dd s,\quad\int^{1}_{0}\bigg|\frac{3s}{2}+W_{s}\bigg|\dd s,\quad\int^{1}_{0}|2s+W_{s}|\dd s.
\end{equation*}
In these four cases, with (\ref{6.1}) the series in (\ref{3.2}) is truncated at $K=6,9,10,11$, respectively, and the series in (\ref{3.8}) at $K=6,7,9,11$, respectively. These choices have been carefully made to ensure an error bound of less than $10^{-4}$ for all input $x$-values (especially at the right tail) for plotting and Figure \ref{fig:2}, Figure \ref{fig:3}, Figure \ref{fig:4}, and Figure \ref{fig:5} show the results.

For the moments the formula (\ref{4.3}) along with (\ref{4.4}) and (\ref{4.5}) is used for obtaining analytic expressions and their corresponding numerical values are all rounded to 10 significant digits. All implementations are done in Mathematica$^\circledR$ (\cite[Wolfram Research, Inc., 2023]{W1}) and results are reported in Table \ref{tab:1},\footnote{Table \ref{tab:1} is the same as \cite[Tak\'{a}cs, 1993, \text{Tab.} 3]{T1}.} Table \ref{tab:2}, Table \ref{tab:3}, and Table \ref{tab:4}.

Importantly, for small to moderate values of $b=\mu/\sigma\neq0$ it usually requires a high machine precision to correctly convert the exact values of some high-order moments to numerical values. Clearly, this challenge is caused by taking differences of large numbers subject to division by irrationals. Somewhat surprisingly, such numerical conversion can easily fail in many other programming languages, including Matlab$^\circledR$, at least in default settings -- check, e.g., the last three rows in Table \ref{tab:2}. To effectively avoid this issue, one may switch to the formula (\ref{4.17}) with proper truncation of the infinite series, as discussed in Section \ref{sec:4}, while the asymptotic estimate (\ref{5.31}) can also come in handy when necessary.

\clearpage % to avoid a bad box

\vspace*{0.7in}

\begin{figure}[H]
  \centering
  \includegraphics[scale=0.45]{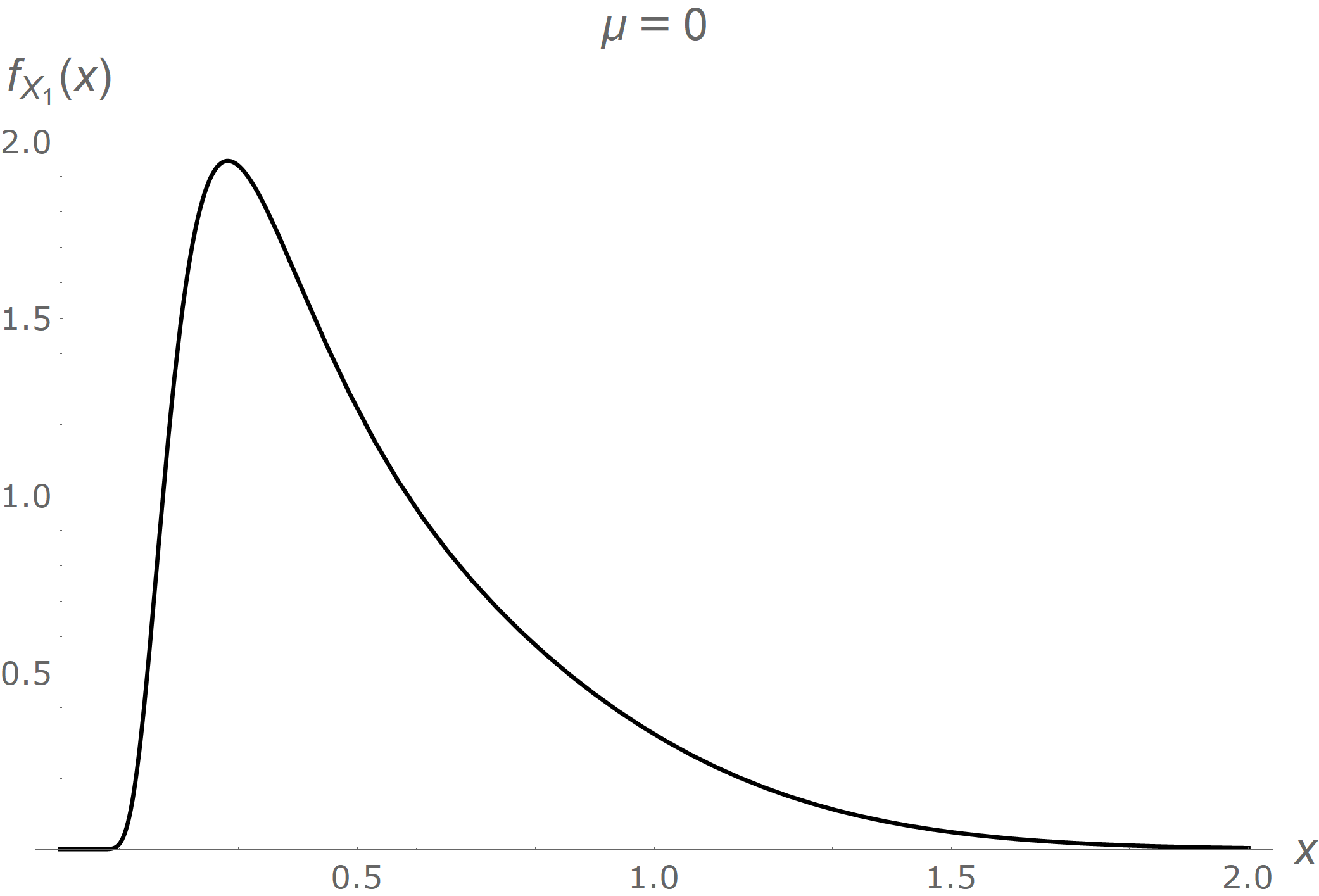}
  \includegraphics[scale=0.45]{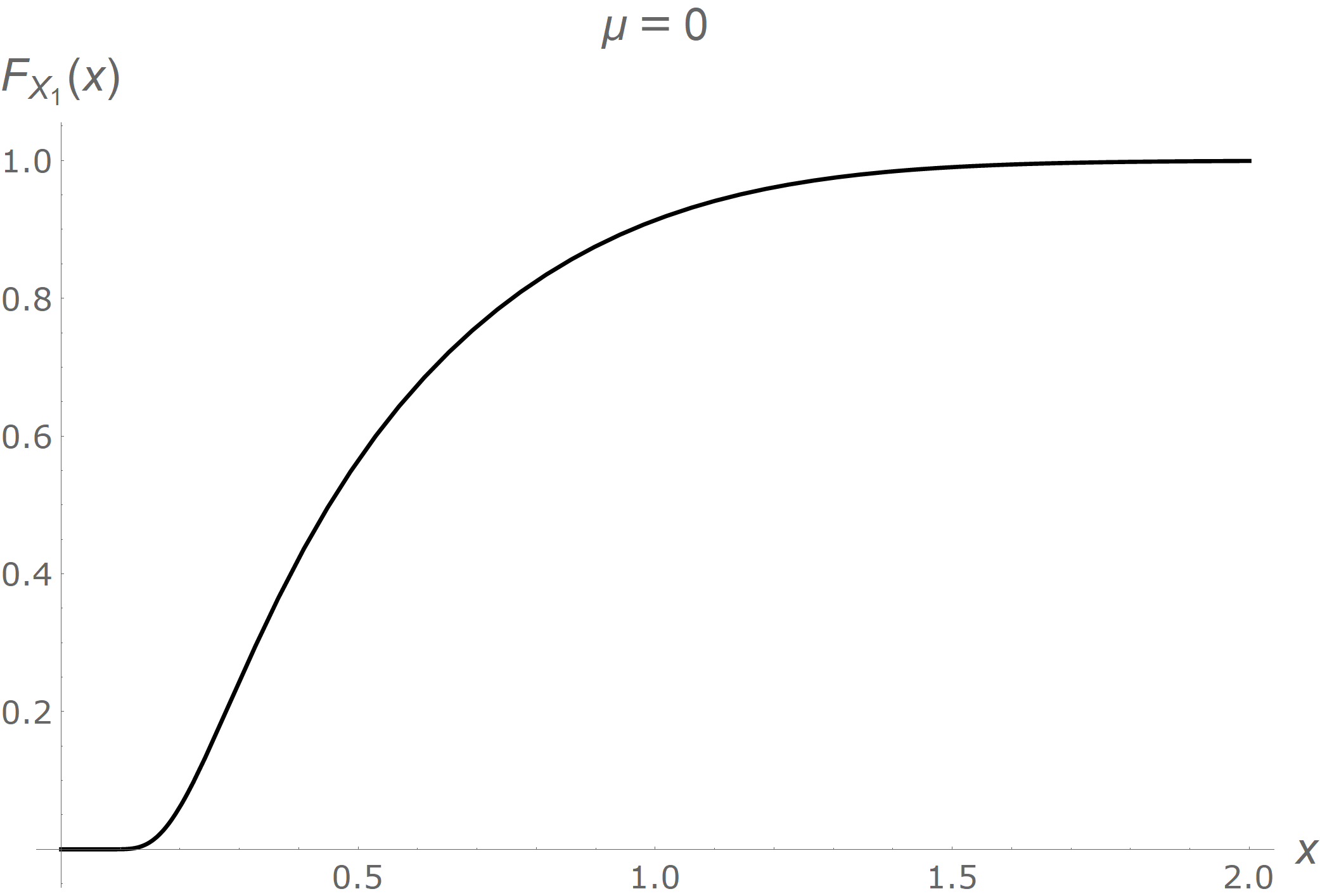}\\
  \caption{Probability density function and cumulative distribution function (I)}
  \label{fig:2}
\end{figure}

\begin{table}[H]\small
  \centering
  \renewcommand{\arraystretch}{1.5}
  \caption{Exact and numerical values of first 10 moments (I)}
  \label{tab:1}
  \begin{tabular}{|c|c|c|}
  \hline
  $n$ & \multicolumn{2}{|c|}{$M_{X_{1}}(n)$, $\mu=0$} \\ \hline
  0 & 1 & 1.000000000 \\
  1 & $\frac{4}{3\sqrt{2\pi}}$ & 0.5319230405 \\
  2 & $\frac{3}{8}$ & 0.3750000000 \\
  3 & $\frac{263}{315\sqrt{2\pi}}$ & 0.3330851420 \\
  4 & $\frac{903}{2560}$ & 0.3527343750 \\
  5 & $\frac{2119}{1980\sqrt{2\pi}}$ & 0.4269488344 \\
  6 & $\frac{37623}{65536}$ & 0.5740814209 \\
  7 & $\frac{11074363}{5250960\sqrt{2\pi}}$ & 0.8413759825 \\
  8 & $\frac{114752519}{86507520}$ & 1.326503395 \\
  9 & $\frac{3845017725821}{688400856000\sqrt{2\pi}}$ & 2.228265881 \\
  10 & $\frac{189970427903}{47982837760}$ & 3.959132823 \\ \hline
  \end{tabular}
\end{table}

\vspace*{0.7in}

\begin{figure}[H]
  \centering
  \includegraphics[scale=0.45]{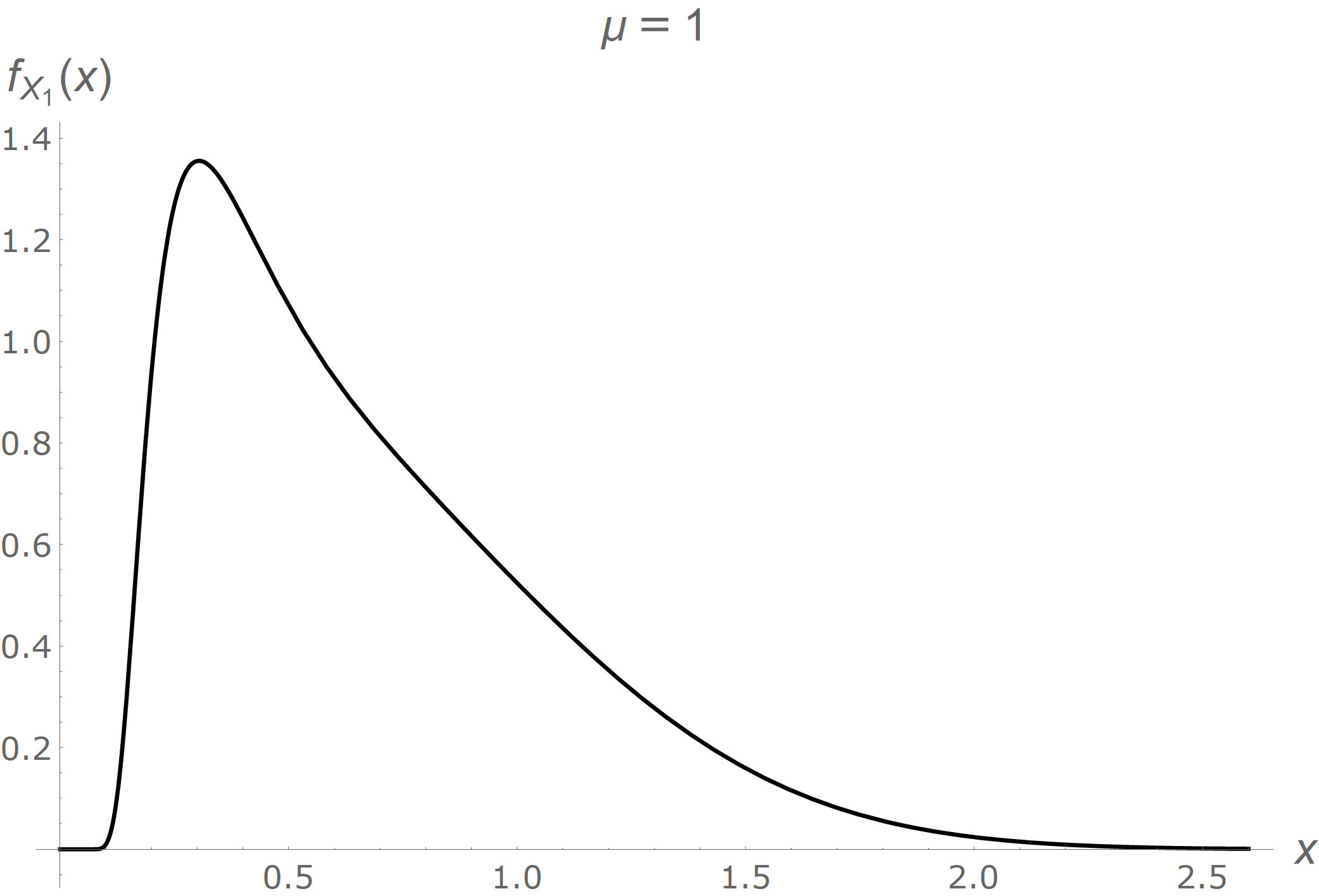}
  \includegraphics[scale=0.45]{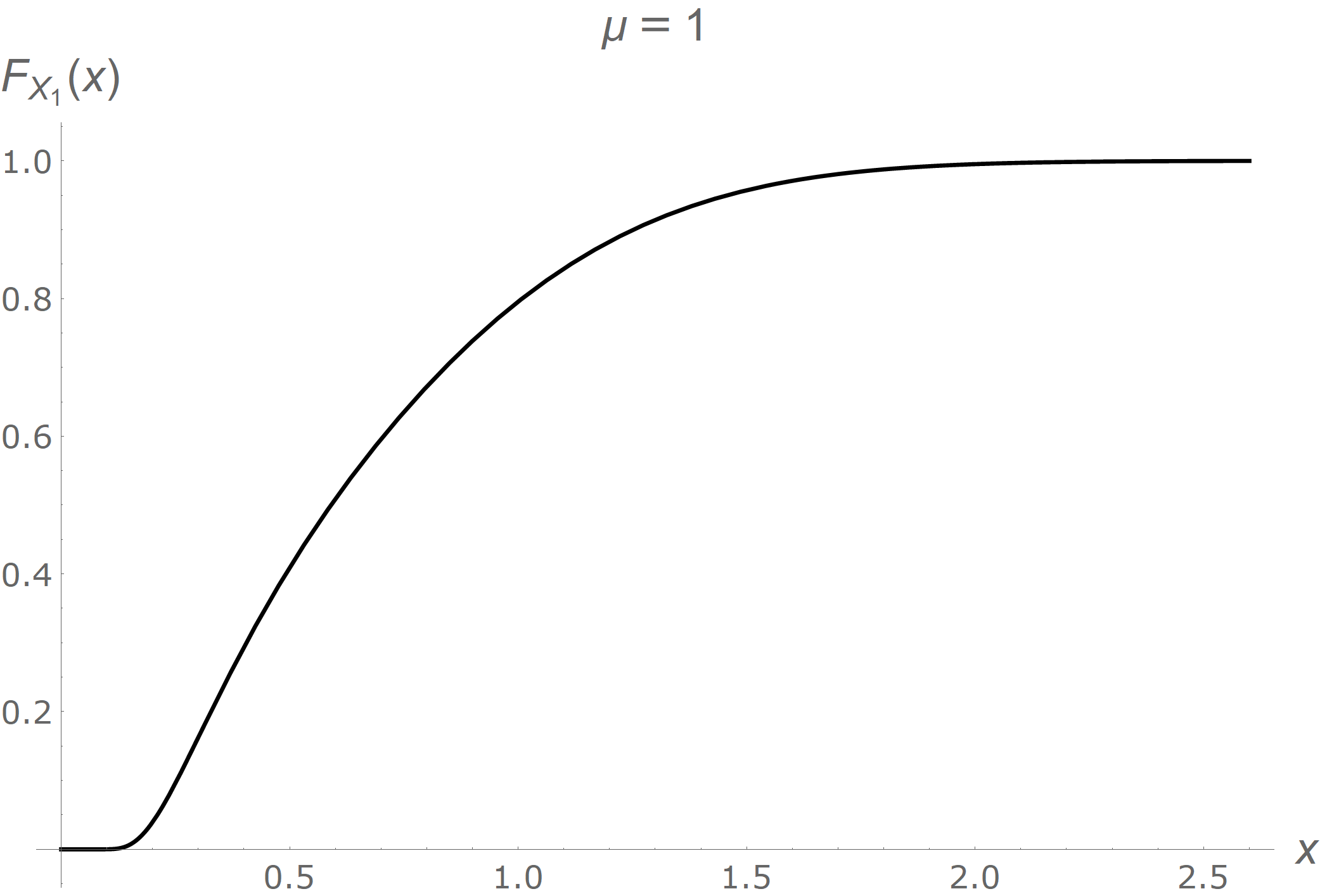}\\
  \caption{Probability density function and cumulative distribution function (II)}
  \label{fig:3}
\end{figure}

\begin{table}[H]\small
  \centering
  \renewcommand{\arraystretch}{1.5}
  \caption{Exact and numerical values of first 10 moments (II)}
  \label{tab:2}
  \begin{tabular}{|c|c|c|}
  \hline
  $n$ & \multicolumn{2}{|c|}{$M_{X_{1}}(n)$, $\mu=1$} \\ \hline
  0 & 1 & 1.000000000 \\
  1 & $\erf\frac{1}{\sqrt{2}}$ & 0.6826894921 \\
  2 & $\frac{73}{12}-\frac{9}{\sqrt{e}}$ & 0.6245573959 \\
  3 & $\frac{939}{4}\erf\frac{1}{\sqrt{2}}-\frac{3297}{5\sqrt{2e\pi}}$ & 0.7058625313 \\
  4 & $\frac{947239}{48}-\frac{2082213}{64\sqrt{e}}$ & 0.9266997049 \\
  5 & $\frac{131820379}{48}\erf\frac{1}{\sqrt{2}}-\frac{30992841}{4\sqrt{2e\pi}}$ & 1.360264768 \\
  6 & $\frac{327870262103}{576}-\frac{1922014837647}{2048\sqrt{e}}$ & 2.180797287 \\
  7 & $\frac{94633197938147}{576}\erf\frac{1}{\sqrt{2}}-\frac{11268507167892207}{24310\sqrt{2e\pi}}$ & 3.761284796 \\
  8 & $\frac{435221973765425411}{6912}-\frac{17008823125880209071}{163840\sqrt{e}}$ & 6.905910669 \\
  9 & $\frac{7916985797851502369}{256}\erf\frac{1}{\sqrt{2}}-\frac{570418135170885612056077}{6537520\sqrt{2e\pi}}$ & 13.39398679 \\
  10 & $\frac{174519586738094763144179}{9216}-\frac{37451989117311817635466828851}{1199570944\sqrt{e}}$ & 27.27829330 \\ \hline
  \end{tabular}
\end{table}

\vspace*{0.7in}

\begin{figure}[H]
  \centering
  \includegraphics[scale=0.45]{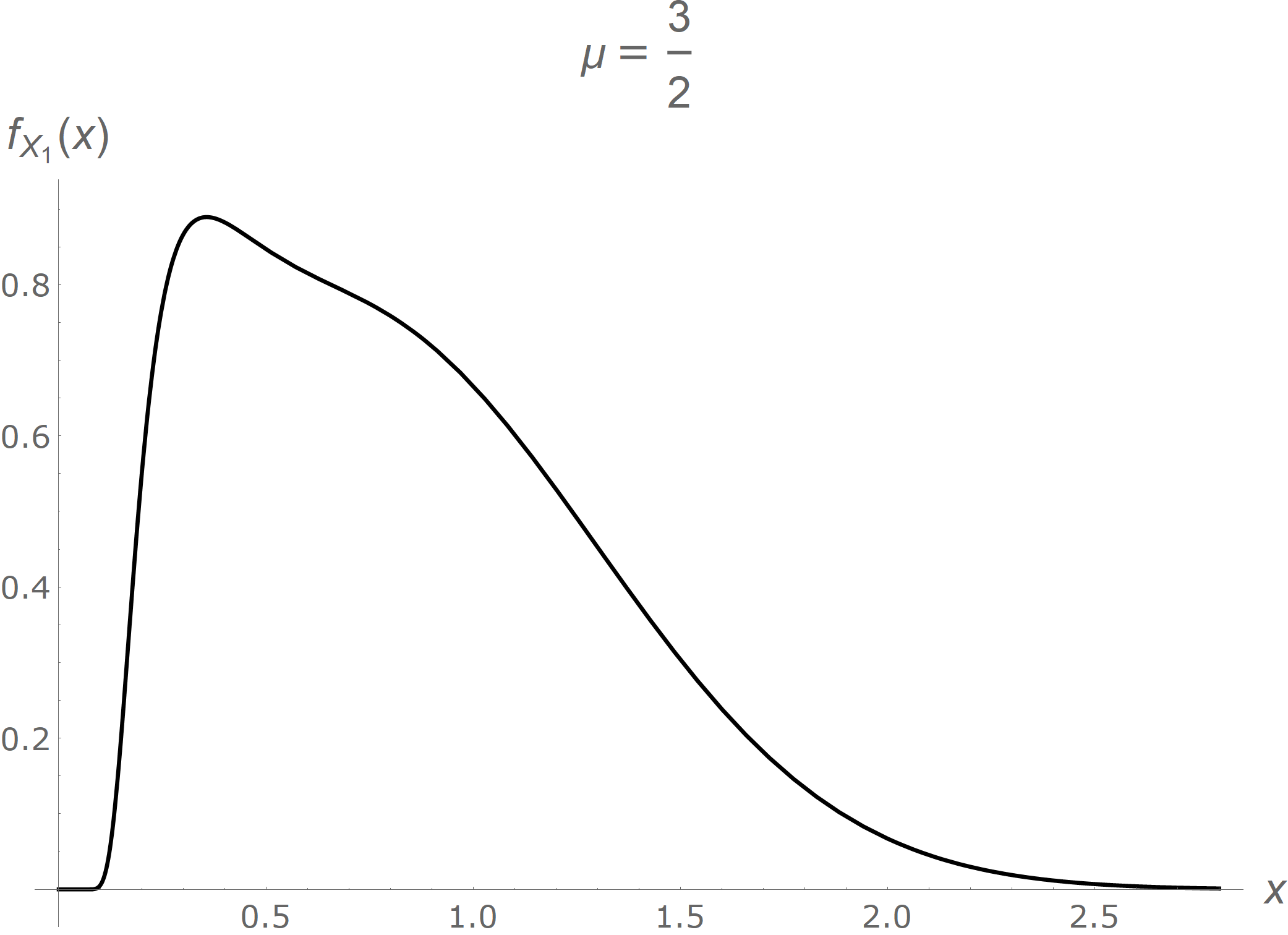}
  \includegraphics[scale=0.45]{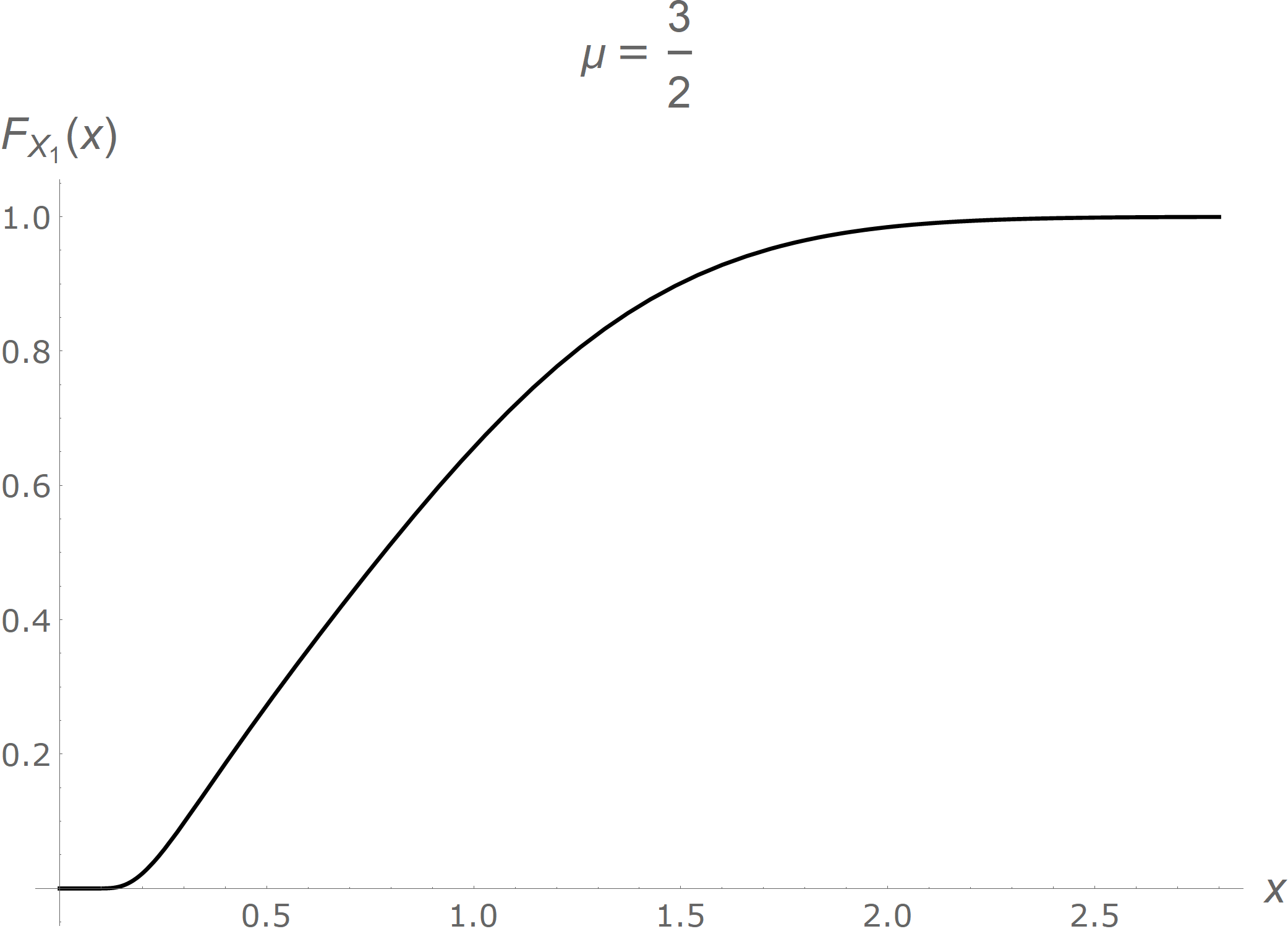}\\
  \caption{Probability density function and cumulative distribution function (III)}
  \label{fig:4}
\end{figure}

\begin{table}[H]\small
  \centering
  \renewcommand{\arraystretch}{1.5}
  \caption{Exact and numerical values of first 10 moments (III)}
  \label{tab:3}
  \begin{tabular}{|c|c|c|}
  \hline
  $n$ & \multicolumn{2}{|c|}{$M_{X_{1}}(n)$, $\mu=\frac{3}{2}$} \\ \hline
  0 & 1 & 1.000000000 \\
  1 & $\frac{97}{108}\erf\frac{3}{2\sqrt{2}}+\frac{5}{9e^{9/8}\sqrt{2\pi}}$ & 0.8500968398 \\
  2 & $\frac{14321}{11664}-\frac{656}{729e^{9/8}}$ & 0.9356522379 \\
  3 & $\frac{2061371}{419904}\erf\frac{3}{2\sqrt{2}}-\frac{4088989}{174960e^{9/8}\sqrt{2\pi}}$ & 1.226260497 \\
  4 & $\frac{4186573451}{45349632}-\frac{197523431}{708588e^{9/8}}$ & 1.818745077 \\
  5 & $\frac{18111847460587}{4897760256}\erf\frac{3}{2\sqrt{2}}-\frac{10086997127497}{408146688e^{9/8}\sqrt{2\pi}}$ & 2.964701664 \\
  6 & $\frac{39760242371105993}{176319369216}-\frac{3827104216015327}{5509980288e^{9/8}}$ & 5.217356259 \\
  7 & $\frac{365610474311453456969}{19042491875328}\erf\frac{3}{2\sqrt{2}} -\frac{2477276832093291558523009}{19288457395384320e^{9/8}\sqrt{2\pi}}$ & 9.795274998 \\
  8 & $\frac{4469320924955467657898617}{2056589122535424}-\frac{2151012108287581567116319}{321342050396160e^{9/8}}$ & 19.45292471 \\
  9 & $\frac{7785311496611354095854216913}{24679069470425088}\erf\frac{3}{2\sqrt{2}} -\frac{23052264721685622550231641951507229}{10924056422790700400640e^{9/8}\sqrt{2\pi}}$ & 40.60370978 \\
  10 & $\frac{152247414041878900221948971378401}{2665339502805909504} -\frac{33530286070371621454356613089942739}{190571774450622529536e^{9/8}}$ & 88.62750339 \\ \hline
  \end{tabular}
\end{table}

\vspace*{0.7in}

\begin{figure}[H]
  \centering
  \includegraphics[scale=0.45]{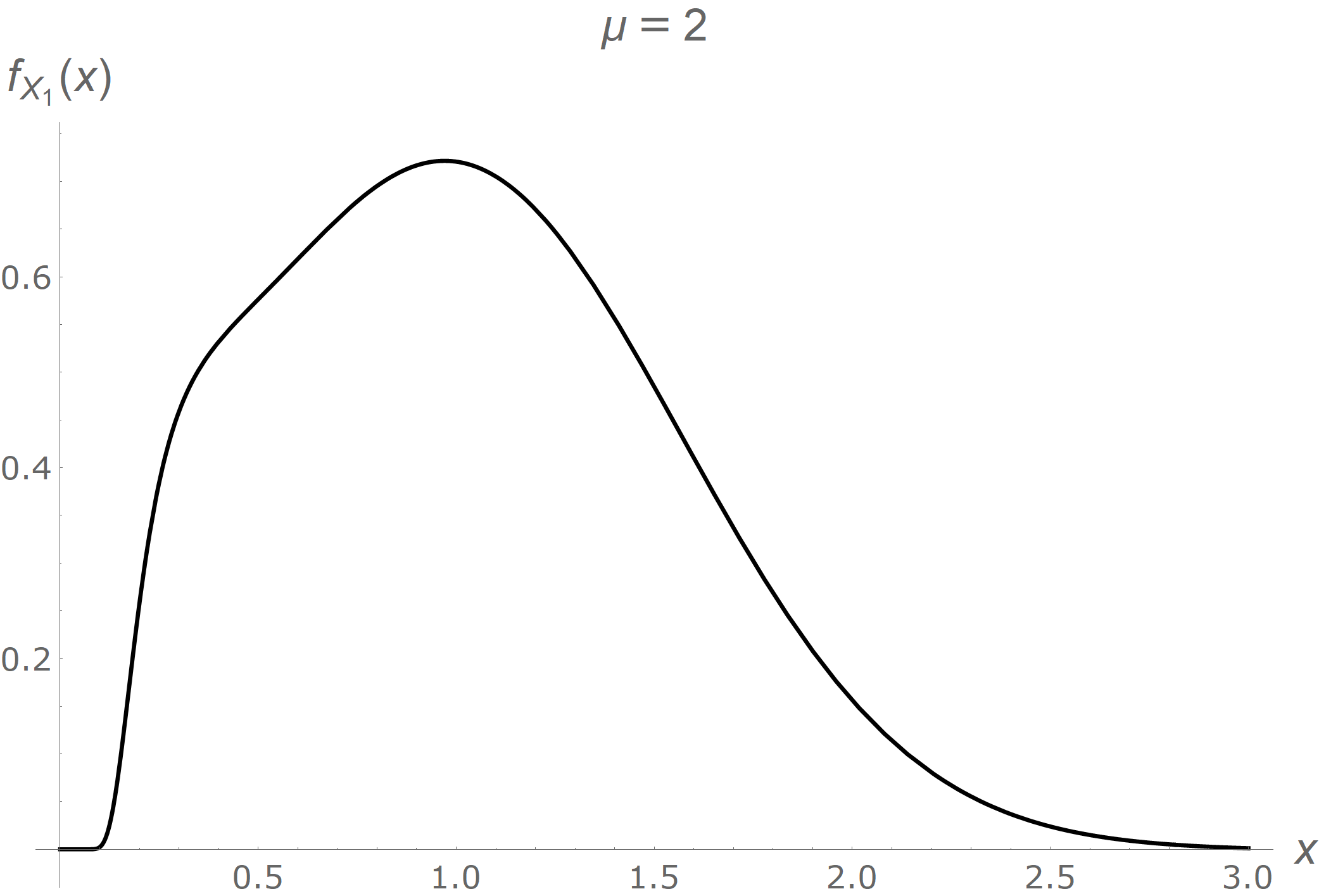}
  \includegraphics[scale=0.45]{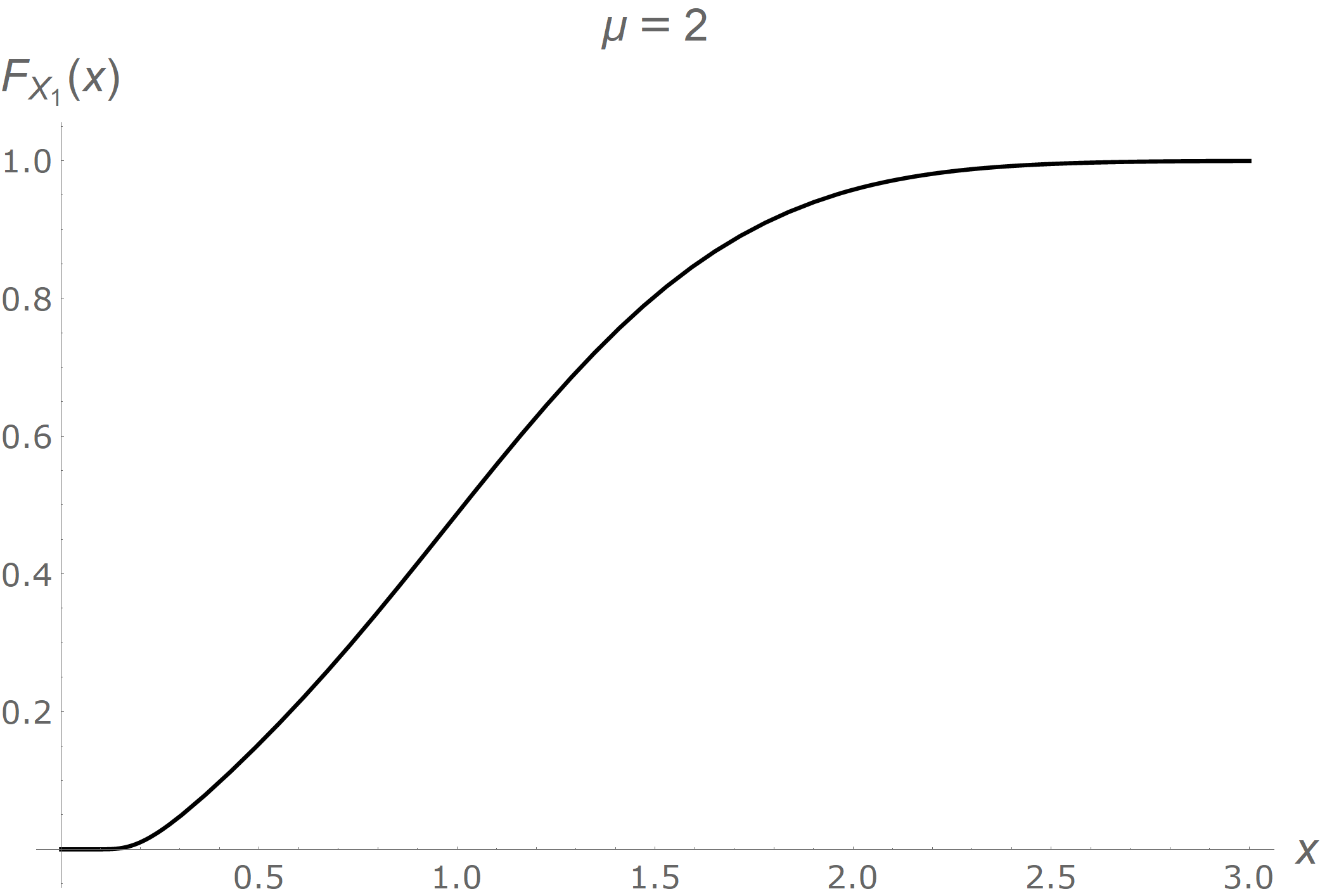}\\
  \caption{Probability density function and cumulative distribution function (IV)}
  \label{fig:5}
\end{figure}

\begin{table}[H]\small
  \centering
  \renewcommand{\arraystretch}{1.5}
  \caption{Exact and numerical values of first 10 moments (IV)}
  \label{tab:4}
  \begin{tabular}{|c|c|c|}
  \hline
  $n$ & \multicolumn{2}{|c|}{$M_{X_{1}}(n)$, $\mu=2$} \\ \hline
  0 & 1 & 1.000000000 \\
  1 & $\frac{17}{16}\erf\sqrt{2}+\frac{3}{4e^{2}\sqrt{2\pi}}$ & 1.054649194 \\
  2 & $\frac{67}{48}-\frac{3}{16e^{2}}$ & 1.370457968 \\
  3 & $\frac{4479}{2048}\erf\sqrt{2}-\frac{1839}{2560e^{2}\sqrt{2\pi}}$ & 2.048717191 \\
  4 & $\frac{29471}{6144}-\frac{21231}{2048e^{2}}$ & 3.393732065 \\
  5 & $\frac{11740501}{393216}\erf\sqrt{2}-\frac{40789601}{98304e^{2}\sqrt{2\pi}}$ & 6.096458892 \\
  6 & $\frac{727680391}{1179648}-\frac{586084689}{131072e^{2}}$ & 11.71447290 \\
  7 & $\frac{815376760909}{37748736}\erf\sqrt{2}-\frac{14584249433910881}{38236323840e^{2}\sqrt{2\pi}}$ & 23.84966824 \\
  8 & $\frac{58026085494757}{56623104}-\frac{79395671621959}{10485760e^{2}}$ & 51.08654013 \\
  9 & $\frac{33552485438869843}{536870912}\erf\sqrt{2}-\frac{41457335045840308194427}{37522579128320e^{2}\sqrt{2\pi}}$ & 114.5088942 \\
  10 & $\frac{22992915872862994063}{4831838208}-\frac{2699457645246760739313}{76772540416e^{2}}$ & 267.4234795 \\ \hline
  \end{tabular}
\end{table}

\clearpage

\section{Concluding remarks}\label{sec:7}

Present a drift component, the absolute-value Brownian integral functional can have a platykurtic-to-leptokurtic distribution. Formulae in terms of rapidly converging series for the distribution functions (Theorem \ref{thm:2} and Theorem \ref{thm:3}) come directly after termwise inverting an asymptotic series representation for the marginalized space Laplace transform (Theorem \ref{thm:1}), with direct implications on small-deviation probabilities (Theorem \ref{thm:5}). The key step in deriving the moment formulae (Theorem \ref{thm:4}) is to obtain asymptotic expansions of the exponential Airy integrals through integration-by-parts and then collect powers via regularizing the incomplete Bell polynomials so that matching terms is possible. Derivation of large-deviation probabilities (Theorem \ref{thm:6}) is based on transforming a double-inverse of the space--time Laplace transform along the lines of \cite[Tolmatz, 2003]{T4}, where the contour deformation is nonetheless realized through a probabilistic argument, in conjunction with various advanced properties of the exponential Airy integrals, which only appear in the presence of drift. Towards this end, we have also developed and applied a generalized Laplace's method for multivariate integral approximation (Lemma \ref{lem:5}). Overall, the results in this paper can be considered a complete characterization of the distribution of the absolute-value integral and they facilitate corresponding applications for both exact and asymptotic computations (refer to Section \ref{sec:1}).

Another class of Brownian integral functionals worth exploring are the sinusoidal type, i.e., functionals of the form $\big(\int^{t}_{0}\cos(\mu s+\sigma W_{s})\dd s\big)_{t\geq0}$. Notably, the centered case ($\mu=0$) subject to constant shifting has been studied in \cite[Pintoux, 2010, \text{Chap.} 4]{P}, using the Feynman--Kac formula; in particular, see \cite[Pintoux, 2010, \text{Prop.} 4.3.2]{P} for a series representation for the space Laplace transform in terms of Mathieu functions. Aside from similar applications mentioned before, such functionals have applications in modeling presumably bounded financial quantities such as interest rates or volatility components (see \cite[Lim and Privault, 2016]{LP} and \cite[Wang and Xia, 2022]{WX}). They are also closely related to the so-called ``imaginary exponential functionals'' considered in \cite[Gredat et al., 2011]{GDL}, which naturally generalize the extensively studied real exponential functionals as well.

On the other hand, further analysis of the distribution -- both exact and asymptotic -- of other power integral functionals of the Brownian motion (with or without drift) except for 1 (as in the present paper) or 2 (as in \cite[Xia, 2020]{X}) will be arduous, simply because the resultant differential equation governing the space Laplace transform from applying the Kac formula does not have solutions expressible in terms of known functions. Speaking of applications, the benefit from such analysis will be marginal as the $\mathrm{L}^{1}$ and $\mathrm{L}^{2}$ cases are arguably the most commonly used.

\end{document}